\documentclass[11pt]{amsart}
\usepackage{url,enumitem, amsopn}   

\newcommand{\cal}{\mathcal}

\newcommand{\nats}{\mbox{$\mathbb N$}}
\newcommand{\ints}{\mbox{$\mathbb Z$}}
\newcommand{\power}{\mbox{$\mathbb P$}}

\newcommand{\spn}{{\mathop{\mathrm{Span}}\nolimits}}

\newcommand{\soc}{{\mathop{\mathrm{Soc}}\nolimits}}
\newcommand{\doc}{{\mathop{\partial\mathrm{oc}}\nolimits}}
\newcommand{\ann}{{\mathop{\mathrm{Ann}}\nolimits}}

\newcommand{\m}{{\mathop{\mathbf{m}}\nolimits}}

\newcommand{\q}{{\mathop{\mathbf{q}}\nolimits}}
\newcommand{\comment}[1]{}

\usepackage{amsmath, amsfonts, amssymb, latexsym, stmaryrd}
\usepackage{graphicx}                   

\def\squarebox#1{\hbox to #1{\hfill\vbox to #1{\vfill}}}
\def\qed{\hspace*{\fill}
        \vbox{\hrule\hbox{\vrule\squarebox{.667em}\vrule}\hrule}\smallskip}

\theoremstyle{plain}
\newtheorem{lemma}{Lemma}[section]
\newtheorem{theorem}[lemma]{Theorem}
\newtheorem{corollary}[lemma]{Corollary}
\newtheorem{proposition}[lemma]{Proposition}

\theoremstyle{definition}
\newtheorem{claim}[lemma]{Claim}
\newtheorem{observation}[lemma]{Observation}
\newtheorem{definition}[lemma]{Definition}
\newtheorem{convention}[lemma]{Convention}
\newtheorem{question}[lemma]{Question}

\newtheorem{example}[lemma]{Example}

\newtheorem*{rmk*}{Remark}
\newtheorem*{rmks*}{Remarks}
\newtheorem*{conventions*}{Conventions}
\newtheorem*{convention*}{Convention}

\def\squareforqed{\hbox{\rlap{$\sqcap$}$\sqcup$}}
\def\qed{\ifmmode\squareforqed\else{\unskip\nobreak\hfil
\penalty50\hskip1em\null\nobreak\hfil\squareforqed
\parfillskip=0pt\finalhyphendemerits=0\endgraf}\fi}

\newlength{\tablength}
\newlength{\spacelength}
\newcommand{\tabstar}{\hspace*{\tablength}}
\newcommand{\spacestar}{\hspace*{\spacelength}}
\def\obeytabs{\catcode`\^^I=\active}
{\obeytabs\global\let^^I=\tabstar}
{\obeyspaces\global\let =\spacestar}
\newenvironment{display}{\begingroup\obeylines\obeyspaces\obeytabs}{\endgroup}
\newenvironment{prog}{\begin{display}\parskip0pt\sf}{\end{display}}

\newcommand{\ra}{\rightarrow}

\newcommand{\into}{\hookrightarrow}
\DeclareMathOperator{\im}{im}
\DeclareMathOperator{\Ass}{Ass}

\author{Geir Agnarsson}
\address{Department of Mathematical Sciences \\ 
George Mason University \\ Fairfax, VA  22030}
\email{geir@math.gmu.edu}

\author{Neil Epstein}
\address{Department of Mathematical Sciences \\ 
George Mason University \\ Fairfax, VA  22030}
\email{nepstei2@gmu.edu}

\title{}
\title{On posets, monomial ideals, Gorenstein ideals and their combinatorics}



\date{\today}

\begin{document}

\begin{abstract}
In this article we first compare the set of elements in the socle of an ideal of
a polynomial algebra $K[x_1,\ldots,x_d]$ over a field $K$ that are not
in the ideal itself with Macaulay's inverse systems of such polynomial algebras
in a purely combinatorial way for monomial ideals, and then develop 
some closure operational properties for the related poset
${{\nats}_0^d}$. We then derive some algebraic propositions of $\Gamma$-graded
rings that then have some combinatorial consequences.
Interestingly, some of the results from this part that uniformly hold for polynomial
rings are always false when the ring is local. 
We finally delve into some direct computations, in relation to a given term
order of the monomials, for general zero-dimensional 
Gorenstein ideals, and we deduce a few explicit observations and results for the inverse
systems from some recent results about socles.
\end{abstract}

\maketitle

\section{Introduction}
\label{sec:intro}

This paper was originally inspired, in part, by some specific questions raised
about what
is similar and what is dissimilar between monomial ideals obtained from
a given antichain of monomials and monomial ideals obtained from their
inverse systems~\cite{Dall}. Most of these questions had known answers, 
but they still spun some novel observations and further
thoughts worth noticing and jotting down. This paper contains these very observations.  

By viewing a given antichain of monomials as generators of the socle of a monomial ideal 
that are not contained in the ideal itself, one can obtain a variety of
combinatorial duality properties purely from a lattice and poset point of 
view. This was done in~\cite{A-Epstein}.
The socle of a commutative ring has been used in a variety
of algebraic ways in the studies of local rings, Cohen-Macaulay rings,
and Gorenstein rings, to name a few~\cite{Bruns-Herzog,Eisenbud,MonAlg}.
This paper is a further study of the socle of an ideal, both from combinatorial
and algebraic points of view.  It can be viewed as a continuation or as an extended 
addendum to~\cite{A-Epstein}. As such, the sections of this paper are largely 
independent, but they all share the common theme of the socle of an ideal, 
particularly that of a monomial ideal. Hence, the focus of the paper is to
demonstrate the interplay between the discrete and combinatorial setting (monomial ideals) and the algebraic setting (inverse systems and general
artinian ideals).

Despite providing an interesting dictionary between commutative algebra and 
the study of differential equations, Macaulay's original inverse system cannot be considered a mainstream modern topic in commutative algebra (though a variant, \emph{Matlis duality}, is important there), let 
alone in combinatorics and the the study of finite posets. In 
algebraic geometry, a modernized version of the inverse system has seen 
some general widely used applications, in particular  providing a Matlis-Macaulay
duality between certain ideals and modules.  This paper is however not 
about algebraic geometry, but rather, and in part, on some related observations regarding 
combinatorial properties of monomial ideals (see for example Observation~\ref{obs:I=J} 
and Corollary~\ref{cor:summary}) and closure operators on them with restricted poset structure (see for example
Theorem~\ref{thm:sqsubseteq-closure-op} and its Corollary~\ref{cor:closure}).
It is also about when these combinatorial properties can be extended to ideals
of polynomial rings in general, and when it makes sense to view them as inverse ideals
of elements in their socles that are not in the ideals 
(see for example Corollary~\ref{cor:uniqueupsetdecomp}, which shows that a corresponding
statement is false for general local rings, and Proposition~\ref{prp:P} and its
Corollary~\ref{cor:ann-iff-monomial}.)
In particular, the {\em docle} of a monomial ideal $I$, denoted $\doc(I)$, 
containing the generators of the
socle of the monomial ideal that are not in the ideal itself
(see Definition~\ref{def:doc(I)})
is a well-defined set for monomial ideals, but not so for ideals of polynomial rings
in general.
For a monomial ideal one can therefore ask:
(a) is the docle $\doc(I)$ in bijective correspondence 
with the generators of the inverse system (IS) for the ideal? Our first objective is to
give a combinatorial and
elementary proof that this is indeed the case in Corollary~\ref{cor:summary}. 
We suspect this to be known, but since we have not seen it stated in a clear way
we begin by providing an elementary proof. A second
question is then:
(b) When does the docle $\doc(I)$ of a non-monomial ideal make sense?
And, when it does make sense, what is the best way to define it? We will see in 
Section~\ref{sec:Gor-inv} that we need to fix a term order, in our case we will
choose the {\em lexicographical} term order (LEX), in order for the docle of any
ideal to be well-defined. This allows us to state our main results Proposition~\ref{prp:P} and its
Corollary~\ref{cor:ann-iff-monomial} in this section.

In order to discuss how our treatment of monomial ideals and their
corresponding posets relate to inverse systems, we need to briefly 
discuss some fundamental definitions and facts about Macaulay's inverse 
systems (IS). We will here, for the most part, use the presentation of the 
inverse system as given in the Lecture notes by Stephan Stolz~\cite{Stolz},
mainly since the treatment there is the most direct and compatible
to our poset discussions that follow. In addition, the treatment there
is close to that of the original theme of Macaulay of treating differential
equations~\cite{Macaulay} (for which the article~\cite{Pommaret} contains many relevant discussions and references).

In what follows we mention a few works, albeit marginally relevant
to our discussion at best, that have influenced the style of writing about IS in this
current article. As stated in~\cite{Stolz}, the main motivation for
writing his lecture notes was the work of Iarrobino and Emsalem who
utilized Macaulay's IS in part:
(i) to study $0$-dimensional subschemes of ${\power}^n$,
(ii) to translate certain bounds on Hilbert functions to vanishing problems of 
inverse systems, and 
(iii) to study the Waring problem (of determining for each given natural number $k$ the smallest 
natural number $s(k)$ such that every natural number can be written as a sum 
of at most $s(k)$ $k$-th powers of natural numbers) 
in~\cite{Iarrobino-Emsalem-I}, \cite{Iarrobino-III} and~\cite{Iarrobino-II} respectively. 
Cho and Iarrobino further studied $0$-dimensional subschemes of ${\power}^n$ 
using IS in~\cite{Cho-Iarrobino}. More related to our work here is 
\cite{VanTuyi-Zanello}, where the authors utilize  
Macaulay's IS to compute the socle of an artinian algebra that is related to 
a given simplicial complex. This directly relates to the motivating work
for this current paper and~\cite{A-Epstein}, namely to the work of Anna-Rose Wolff in her 
MS Thesis~\cite{Anna-Rose-Thesis} where she analyzed the survival complex of $R/I$, 
where $I$ is a zero-dimensional monomial ideal. There the {\em survival complex} is a 
simplicial complex whose vertices are the monomials of $R$ that are not in $I$, where 
a simplex consists of a set of monomials whose product is not in $I$. She 
showed~\cite[18, Prop.~2.2.1]{Anna-Rose-Thesis}
that the truly isolated points of that complex correspond to the monomial basis of the socle 
of $R/I$, that is the docle $\doc(I)$. 

Other useful applications of Macaulay's IS can be found in the literature on
algebraic geometry, For example, a clever use of IS in classical algebraic 
geometry is applied in~\cite{Park-Stanley-Zanello}. Using a well-known 
one-to-one inclusion reversing correspondence, also known as
{\em Macaulay's correspondence} or 
{\em Macaulay-Matlis duality}~\cite[Thm 21.6]{Eisenbud} one can describe
artinian (zero-dimensional) Gorenstein algebras over fields, something we
will come back to later on in this article. This correspondence is 
used to present structure theorems for certain artinian Gorenstein 
algebras~\cite{Elias-Rossi-Isom}. These same authors extend their studies,
using IS, in~\cite{Elias-Rossi-IS-Gorenstein}. It should be noted that the
notion of Macaulay's IS given in~\cite[p.~526]{Eisenbud} is not the classical
one, but rather a polished equivalent version of it that relies less on the
characteristic of the ground field $K$ of the $K$-algebras in question. 
As mentioned above, and as we will see shortly, we have instead chosen to use the 
classical presentation of Macaulay's IS as differential operators.
Our main use of Macaulay's IS will be to obtain an if and 
only if statement in Corollary~\ref{cor:ann-iff-monomial}, the proof of 
which depends on the classical version of Macaulay's IS.

The rest of this article is organized as follows:

In Section~\ref{sec:defs}, we recall some basic concepts for posets and state
a few claims and observations that we will use later in the paper. We then
define the inverse system (IS) in the way we will be using it in the paper,
as well as the dual concept of the inverse ideal. 

In Section~\ref{sec:socle-IS}, we discuss monomial ideals in $R = K[x_1,\ldots,x_d]$,
the monoid $[x_1,\ldots,x_d]$, and the socle of a monomial ideal of $R$, all from a
combinatorial lattice point of view. Given a set $M$ of non-comparable monomials
of $R$, we start by demonstrating in a combinatorial and elementary way that the 
inverse ideal $I$ corresponding to $M$ is the unique 
zero-dimensional monomial ideal $I$ with its docle $\doc(I) = M$ 
(see Observation~\ref{obs:I=J}). We then derive in
a purely combinatorial way some dual properties of the inverse system and inverse ideals.
These dual properties have algebraic analogies. 

In Section~\ref{sec:closure-op}, we discuss a sort of closure
operator on monomial ideals, or equivalently on upsets of a corresponding poset,
be it the multiplicative monoid $[x_1,\ldots,x_d]$ or the additive monoid 
${\nats}_0^d$, and we derive some of its properties. 

In Section~\ref{sec:alg-observations}, we present some properties of the docle
of a monomial ideal that are derived in purely algebraic ways. In particular,
we consider a $\Gamma$-graded ring $R$ such that the ideal generated by the
homogeneous non-units is not all of $R$ and 
we present
an algebraic proof of the fact that if $I\subseteq J$ are zero-dimensional monomial
ideals with $\doc(I)\subseteq\doc(J)$ then $I = J$ 
(see Proposition~\ref{prp:idealanddoc}). This has then a purely combinatorial consequence 
for upsets in a poset $P$ in Corollary~\ref{cor:poset-idealanddoc}. We then prove an
algebraic generalization of what we know holds for monomial ideals with common socle
in Theorem~\ref{thm:expandtomprimary}, which then allows us to present a given upset 
$U$, with a non-empty $\doc(U)$, in a unique way as an intersection $U = V\cap W$ of 
two upsets where $\doc(W)$ is empty and $V$ is {\em cofinite} (that is $P\setminus V$ 
is finite and corresponds to a zero-dimensional monomial ideal, see next section) 
in Corollary~\ref{cor:uniqueupsetdecomp}. Finally,
it is interesting to note that 
the \emph{uniqueness} of
the mentioned intersection in Corollaries~\ref{cor:uniqueidealdecomp} 
and~\ref{cor:uniqueupsetdecomp} is \emph{always 
false} in the case where the ring is local \cite{HRS-emb1}.

In Section~\ref{sec:Gor-inv}, we use a known form for homogeneous zero-dimensional 
Gorenstein ideals of the polynomial ring $R = K[x_1,\ldots,x_d]$ to first
prove some observations about such ideals, where we we heavily utilize 
computations with respect to the given term order LEX. This allows us to obtain
some novel explicit results on the presentations of such a homogeneous zero-dimensional
Gorenstein ideal, as the inverse ideal of an explicit polynomial we call
the {\em $d;k$-antipodal polynomial}, in Proposition~\ref{prp:P}. This, in return,
will yield an if-and-only-if statement on when exactly we have the property
from Observation~\ref{obs:I=J} for homogeneous zero-dimensional Gorenstein 
ideals: namely the ideal is the inverse ideal of the sole element of its
docle if and only if the ideal is monomial. 

In Section~\ref{sec:universal}, we show that the form
of our homogeneous zero-dimensional Gorenstein ideal of the polynomial ring
$R = K[x_1,\ldots,x_d]$ presented in Section~\ref{sec:Gor-inv} has a more 
universal form in terms of analytic functions and not merely polynomials.
That is, the inverse ideal which
yields the given homogeneous zero-dimensional Gorenstein ideal can be 
chosen relatively freely, as described in Proposition~\ref{prp:any-f}.

Finally, in Section~\ref{sec:some-q}, we present a brief summary and present
a few questions for further study.

\section{Basic definitions and properties}
\label{sec:defs}

The set of integers will be denoted by $\ints$, the
set of positive natural numbers by $\nats$, and the set of
non-negative integers $\nats\cup\{0\}$ by ${\nats}_0$. For each
$n\in\nats$ we let $[n] := \{1,\ldots,n\}$. 

For a poset $(P,\leq)$ recall that $C\subseteq P$
is a {\em chain} if $C$ forms a linearly or totally ordered
set within $(P,\leq)$. A subset $N\subseteq P$ is an {\em antichain}
if no two elements in $N$ are comparable
in $(P,\leq)$. We call a subset $U\subseteq P$ an {\em upset} of $P$ if $x\geq y\in U \Rightarrow x\in U$.
We call a subset $D\subseteq P$ a {\em downset} of $P$ if $x\leq y\in D \Rightarrow x\in D$.
These definitions directly imply the following.
\begin{claim}
\label{clm:up-down}
If $P$ is a poset with a set partition $P = U\cup D$,
then $U$ is an upset if and only if $D$ is a downset.
\end{claim}
For a subset $A\subseteq P$, 
let $U(A) := \{x\in P : x\geq a \mbox{ for some } a\in A\}$
be the {\em upset generated by $A$}, and 
$D(A) := \{x\in P : x\leq a \mbox{ for some } a\in A\}$
the {\em downset generated by $A$}. If $A = \{a_1,\ldots,a_n\}$
is finite, then we will write $U(a_1,\ldots,a_n)$ (resp.~$D(a_1,\ldots,a_n)$)
for $U(A)$ (resp.~$D(A)$) and say that $U$ (resp.~D) is {\em finitely generated}.

A poset $P$ satisfies the {\em ascending chain condition (ACC)} if there 
is no infinite sequence $U_1\subset  U_2 \subset U_3 \subset \cdots$ of upsets in $P$. 
Similarly to the condition of a ring being Noetherian we have the following.
\begin{lemma}
 \label{lmm:ACC}
A poset $(P,\leq)$ satisfies the ACC if and only if each 
upset is finitely generated.
\end{lemma}
For the convenience of the reader we include a proof:
\begin{proof}
Suppose each upset in $P$ is finitely generated and let
$U_1\subseteq U_2\subseteq U_3\subseteq \cdots$ be a chain of upsets in $P$. 
Then the union $U = \bigcup_{i\geq 1} U_i$ is an upset in $P$ and 
therefore finitely generated, say $U = U(a_1,\ldots,a_n)$. It follows that there is an 
$m$ such that $\{a_1,\ldots,a_n\}\subseteq U_m$ in which case we have 
$U_m = U_{m+1} = U_{m+2} = \cdots$ which implies that that the chain is 
stationary.

Conversely, suppose $U$ is an upset in $P$ that is not finitely generated.
In this case there is a sequence $(a_i)_{i\geq 1}$
of elements in $U$ such that $U(a_1)\subset U(a_1,a_2)\subset U(a_1,a_2,a_3)\subset\cdots\subset U$
is a strictly increasing sequence of upsets in $U$ and hence in $P$. 
\end{proof}

Let $P$ be a poset satisfying the ACC and $U$ an upset of $P$. Since
$U$ is finitely generated, say $U = U(a_1,\ldots,a_n)$, we can assume the
set $\{a_1,\ldots,a_n\}$ to be minimal in the sense that no proper subset
of this set generates $U$. In this case, $\{a_1,\ldots,a_n\}$ forms an antichain in $P$.
Hence, we can always assume the generators of an upset form an antichain in $P$ if
it satisfies the ACC. Clearly, this antichain is uniquely determined by $U$ and consists
of the minimal elements of $U$.

We say that set $X\subseteq P$ is {\em cofinite} if $P\setminus X$ is
a finite subset of $P$. So, an upset $U$ of $P$ is cofinite if and only if 
the downset $D = P\setminus U$ is finite. In that case, $D$ is clearly also
finitely generated, say $D = D(a_1,\ldots,a_n)$, and assuming the generating set
$\{a_1,\ldots,a_n\}$ is minimal, it will form an antichain. This antichain
is uniquely determined by the downset $D$ as it consists of the maximal elements
of $D$.

Let $R = K[x_1,\ldots,x_d]$ denote the polynomial ring
over a field $K$ in $d$ variables. By the {\em socle} of an 
ideal $I\subseteq R$ with respect to the maximal ideal 
$\m = (x_1,\ldots,x_d)$ of $R$, we will mean the ideal 
\[
\soc(I) = \soc_{\m}(I) := (I:\m) = \{f\in R: x_if\in I \mbox{ for each }i\}.
\]
Note that for a monomial ideal $I\subseteq R$, the set
$\soc(I)\setminus I$ contains monomials $a\in R$ such that
(i) $a\not\in I$ and (ii) $x_ia\in I$ for every $i\in[d]$,
and vice versa, if a monomial $a\in R$ satisfies these two
conditions, then $a\in\soc(I)\setminus I$.
\begin{definition}
\label{def:doc(I)}
For a monomial ideal $I$, let $\doc(I)$ denote the {\em docle} of $I$, defined to be the set of monomials in $\soc(I)\setminus I$.
\end{definition}

Note that ${\ints}^d$ has a natural partial order $\leq'$
where for $\tilde{a} = (a_1,\ldots,a_d)$ and $\tilde{b} = (b_1,\ldots,b_d)$
we have $\tilde{a}\leq'\tilde{b} \Leftrightarrow a_i\leq b_i$ for
each $i\in \{1,\ldots,d\}$. We will from now on use ``$\leq$'' for this 
partial order ``$\leq'$''.
The usual componentwise addition $+$ that makes $({\ints}^d,+)$ an
abelian group respects this partial order $\leq$. The map 
$\tilde{p} = (p_1,\ldots,p_d)\mapsto x_1^{p_1}\cdots x_d^{p_d}
= {\tilde{x}}^{\tilde{p}}$ is an isomorphism between
the additive monoid ${\nats}_0^d$ and the multiplicative
monoid $G = [x_1,\ldots,x_d]$. This map is also
an order isomorphism from the lattice $({\nats}_0^d,\leq)$
to $[x_1,\ldots,x_d]$ as a lattice given by divisibility and
so it is an order preserving monoid isomorphism.
Via this isomorphism
we have a bijective correspondence between monomial ideals of $R$ and
upsets of the lattice ${\nats}_0^d$.

Note that the elements in the docle $\doc(I)$ (see Definition~\ref{def:doc(I)}) are 
the maximal monomials from $G$ with respect to divisibility that are not in $I$,
that is $\doc(I) = \max(G\setminus I)$. With these motivating terminologies 
and definitions from algebra, the corresponding definitions for lattices
and posets are natural. For a poset $P$ and an upset $U\subseteq P$ we let
$\doc(U) := \max(P\setminus U)$. 
By Claim~\ref{clm:up-down}, $P\setminus D(M)$ is an upset of $P$. 
The following is also clear.

\begin{claim}
\label{clm:M=doc}
If $P$ is a poset and $U\subseteq P$ is an upset then
(i) $M = \doc(U)$ is an antichain of $P$,
(ii) $D(M)\subseteq P\setminus U$ are both downsets of $P$, and
(iii) $\max(D(M)) = M = \max(P\setminus U)$.
\end{claim}

\begin{rmk*}
A subset $A\subseteq G = [x_1,\ldots,x_d]$ of monomials
generates an upset $U(A)$ of the monoid $G$, when $G = (G,|)$ is
viewed as the lattice given by divisibility (which is, as we just mentioned,
order isomorphic to the lattice $({\nats}_0^d,\leq)$), and it also generates
a monomial ideal
$I(A)$ of $R = K[x_1,\ldots,x_d]$. With our conventions above the sets
$\doc(U(A))\subseteq G$ and $\doc(I(A))\subseteq R$ are equal sets dead on;
they contain the same set of monomials from $G\subseteq R$. This lattice
$(G,|)$, from the multiplicative monoid $G$, was first studied by
Gasharov, Peeva and Welker in~\cite{Gasharov-Peeva-Welker}
\end{rmk*}

We briefly describe {\em Macaulay's inverse system (IS)} in modern
notation that we will be using. As mentioned earlier, we will mostly follow 
the notation of Stephan Stolz in his Lecture Notes in the section on the 
inverse system~\cite{Stolz} which
is modern, yet contains most of the original flavor.

Let $R = K[x_1,\ldots,x_d]$ and $S = K[y_1,\ldots,y_d]$ be polynomial 
rings, over a field $K$ of characteristic zero, on disjoint sets 
of variables of the same cardinality $d$. 
The ring $R$ acts on $S$ in a $K$-linear fashion
where $x_i\circ y_j := (\partial/\partial y_i)y_j = \delta_{i\/j}$,
the Kronecker delta function, and where this action extends to 
all of $R$ in the natural way, using the standard formal
method of differentiation. This action makes $S$ into an $R$-module. As both $R = \bigoplus_{n\geq 0}R_n$  and
$S = \bigoplus_{n\geq 0}S_n$ are graded $K$-algebras, we then get
a restricted action of $R_i \times S_j \rightarrow S_{j-i}$ for each pair $i,j$. 
This action of $R$ on $S$ satisfies the following
\begin{equation}
\label{eqn:xoy}
\tilde{x}^{\tilde{p}}\circ \tilde{y}^{\tilde{q}} = \left\{
\begin{array}{ll}
0 & \mbox{ if } \tilde{p}\not\leq\tilde{q}, \\
\prod_{i=1}^d\frac{(q_i)!}{(q_i-p_i)!}\tilde{y}^{\tilde{q}-\tilde{p}}
 & \mbox{ if } \tilde{p}\leq\tilde{q}.
\end{array}
\right.
\end{equation}

\begin{example}
  \label{exa:diff}
Consider the case $d=2$. Since $f = x_2^2 - x_1x_2\in K[x_1,x_2]$ is
homogeneous of degree $2$ and $g' = y_1^2y_2\in S = K[y_1,y_2]$ homogeneous
of degree $3$, we see that $(f,g')\in R_2 \times S_3$ and
$(f,g')\mapsto f\circ g'\in S_1$. According to the rules of differentiation,
\[
f\circ g' =  (x_2^2 - x_1x_2)\circ(y_1^2y_2)
= \left(\left(\frac{\partial}{\partial y_2}\right)^2 -
\left(\frac{\partial}{\partial y_1}\right)
\left(\frac{\partial}{\partial y_2}\right)\right)y_1^2y_2
= -2y_1,
\]
\end{example}

For a subset $X\subseteq R$ let 
$\ann(X) = \{g'\in S : f\circ g' = 0 \mbox{ for each } f\in X\}$
be the set of elements of $S$ that are annihilated by each element 
in $X$. If $I = I(X)$ is the ideal of $R$ generated by $X$, then clearly
$\ann(X) = \ann(I(X))$. The set $\ann(X)$ is an $R$-submodule of $S$.
Likewise, for a subset $X'\subseteq S$ let
$\ann(X') = \{f\in R : f\circ g' = 0 \mbox{ for each } g'\in X'\}$
be the set of elements in $R$ that annihilate each element 
in $X'$. If $R(X') = R\circ X'$ is the $R$ submodule of $S$ generated by $X'$
then clearly $\ann(X') = \ann(R(X'))$. The set $\ann(X')$ is
an ideal of $R$.

\begin{convention}
\label{conv:apostrophe}
As indicated in the above paragraph, we will, when
appropriate, denote the elements, subsets or ideals of $R$ by $X$,
with no apostrophe, and denote the corresponding elements, subsets or submodules of
$S$ by $X'$, with an apostrophe. Many times, if $X$ is given, then $X'$
will denote the corresponding subset in $S$ by replacing each $x_i$ in $X$ by
the variable $y_i$.
\end{convention}

\begin{definition}
\label{def:IS}
For an ideal $I\subseteq R$, we let $I^{-1} := \ann(I)\subseteq S$ 
and call it the {\em inverse system}\footnote{Sometimes in the 
  literature $I$ is assumed to be a homogeneous ideal when the inverse system
  is defined. We should mention that the notion of the inverse ideal is a non-conventional name
for the dual notion of the inverse system.} of $I$. Dually, for
an $R$-submodule $A'$ of $S$ we let ${A'}^{-1} := \ann(A')\subseteq R$
and call it the {\em inverse ideal} of $A'$.
\end{definition}

\begin{rmk*}
We should note that sometimes, as
in~\cite[p.~526]{Eisenbud}, Macaulay's IS is defined as an
action of  $R = K[x_1,\ldots,x_d]$ on $R^{(-1)} := K[x_1^{-1},\ldots,x_d^{-1}]$ by the usual
multiplication of monomials if the result actually lies in $R^{-1}$ 
and zero otherwise:
\begin{equation}
\label{eqn:yoy}
\tilde{x}^{\tilde{p}}\circ (\tilde{x}^{-1})^{\tilde{q}} = \left\{
\begin{array}{ll}
0 & \mbox{ if } \tilde{p}\not\leq\tilde{q}, \\
(\tilde{x}^{-1})^{\tilde{q}-\tilde{p}}
 & \mbox{ if } \tilde{p}\leq\tilde{q}.
\end{array}
\right.
\end{equation}
The action in (\ref{eqn:yoy}) is the same action as in (\ref{eqn:xoy})
except that the coefficients have been removed.
Although (\ref{eqn:yoy}) is
not the explicit differential action, it is more compatible
with ordinary $K$-algebras, as opposed to those with a divided power structure
added.
\end{rmk*}

\section{The socle and the inverse system for monomial ideals}
\label{sec:socle-IS}

For monomial ideals of $R = K[x_1,\ldots,x_d]$ it is particularly
simple to describe their inverse systems in $S$, and vise versa, 
to obtain the ideal in $R$ from 
a given set of monomials in $S$ when viewed as the inverse system of 
an ideal. The inverse system in $S$ of an ideal $I$ is generated as an $R$-module by a set of monomials 
if and only if $I$ is a monomial ideal. This is in great part 
because we can view monomials 
purely from the point of a monoid $G = [x_1,\ldots,x_d]$, with its  
lattice structure where the partial order is given by divisibility.
This will be the main approach in this section.

With Definition~\ref{def:doc(I)} in mind, consider the following 
theorem from~\cite{A-Epstein}.

\begin{theorem}
\label{thm:eye-catcher}
For any non-empty set $M$ of non-comparable monomials  of
$R = K[x_1,\ldots,x_d]$ and a monomial $m$ of $R$ with  
$M\subseteq (m)$, there is a unique monomial ideal 
$I\subseteq (m)$ with $\doc(I) = M$ and $(m)/I$ finite
dimensional as a vector space over $K$.

In particular, for any non-empty set $M$ of non-comparable monomials  of
$R = K[x_1,\ldots,x_d]$, there is a unique zero-dimensional monomial ideal
$I$ of $R$ with $\doc(I) = M$.
\end{theorem}

\begin{example}
\label{exa:d=2}
To illustrate Theorem~\ref{thm:eye-catcher}, consider the case $d=2$
so $R = K[x,y]$ and the set $M = \{x^2y^3,x^5y\}$ of non-comparable
monomials of $R$.

For the monomial $m = xy\in R$ we have $M\subseteq (m)$
and the unique monomial ideal $I\subseteq (m)$ with $\doc(I) = M$
and $(m)/I$ of finite dimension over $K$ is here given by
$I = (xy^4, x^3y^2,x^6y)$. In this case we have $\dim_K((m)/I) = 9$. 

For the monomial $1\in R$ we clearly have $M\subseteq (1) = R$
and the unique zero dimensional ideal $I\subseteq R$ with $\doc(I) = M$
is here given by $I = (y^4, x^3y^2,x^6)$. In this case we have
$\dim_K(R/I) = 18$.
\end{example}

With the setup of Theorem~\ref{thm:eye-catcher}, let 
$M\subseteq R = K[x_1,\ldots,x_d]$ be a set of non-comparable monomials
and let $M'\subseteq S = K[y_i,\ldots,y_d]$ be the analogous set of monomials
in $S$, relabeled in terms of the variables $y_i$s instead of $x_i$s.
By the above Theorem~\ref{thm:eye-catcher} there is a unique zero-dimensional
monomial ideal $I$ of $R$ with $\doc(I) = M$. We will now present a different
description of this ideal $I$ as an inverse ideal. 

Let $J = \ann(M')$ be the inverse ideal of the set $M'$
of monomials in $S$ corresponding to $M$. We will show, in a combinatorial way,
that the ideals $J$ and $I$ are identical. As we have noted before,
$J = \ann(R(M'))$, and since $R(M') = R(D(M'))$, we have
$J = \{f\in R : f\circ a' = 0 \mbox{ for all } a'\in D(M')\}$, where
$D(M')$ is the downset of all monomials in $G' = [y_1,\ldots, y_d]$ that divide
some monomial in $M'$. In fact, for any set of monomials $X'$ of $S$, the
$R$ module of $S$ generated by $X'$ is always generated as a $K$-vector space
by the downset $D(X')$ of $G'$, and so $\ann(M') = \ann(D(M'))$.
We will now argue that $J$ is a monomial ideal of $R$ and is equal to $I$.

{\sc The case $I\subseteq J$:}
We note here that the monomials in $I$ form an upset $U(I)$ in 
the monoid and lattice $G = [x_1,\ldots,x_d]$ with respect to divisibility, and
that $I$ is spanned by $U(I)$ as a $K$-vector space.
Since $M\subseteq G\setminus U(I)$ which is a downset in $G$, it follows from 
Claim~\ref{clm:up-down} that
no monomial in $I$ can divide a monomial in $M$, as it is an element 
in the upset $U(I)$. Hence, by (\ref{eqn:xoy}), when a monomial in $I$ 
acts on a monomial in $M'$ the result will be zero and so $I\subseteq J$. 

{\sc The case $J\subseteq I$:}
To show the containment  $J\subseteq I$ we begin with a lemma.
\begin{lemma}
\label{lmm:ann-mon=mon}
Let $f = \sum_ic_ia_i \in R$ where $c_i\in K$ and each $a_i\in G$ is a monomial,
and $a'\in G'\subseteq S$ be a monomial. If $f\in\ann(a')$ then
$a_i\in\ann(a')$ for each $i$.

Let $g' = \sum_ic_ia_i' \in S$ where $c_i\in K$ and each $a_i'\in G'$ is a monomial,
and $a\in G\subseteq R$ be a monomial. If $g'\in\ann(a)$ then
$a_i'\in\ann(a)$ for each $i$.
\end{lemma}
\begin{proof}
For the first assertion assume $f\circ a' = 0$ and that $a_i\circ a' = 0$
does not hold for all $i$. We may assume $f\in R$ is minimal as a sum,
so that $a_i\circ a' \neq 0$ for each $i$ and that the $a_i$ are distinct.
Writing $a_i = {\tilde{x}}^{\tilde{p}_i}$ for each $i$ and
$a' = {\tilde{y}}^{\tilde{q}}$ we therefore have by
(\ref{eqn:xoy}) that $\tilde{p}_i\leq\tilde{q}$ for each $i$. 
Since the monomials $a_i$ are distinct, so are their exponent vectors
$\tilde{p}_i$ in ${\nats}_o^d$. Therefore the exponent vectors 
$\tilde{q} - \tilde{p}_i$ in ${\nats}_o^d$ are also distinct. Hence by (\ref{eqn:xoy}), $f\circ a' = \sum_i c_i(a_i\circ a')$ is a
$K$-linear combination of distinct monomials in $S$ having the
same number of nonzero terms as in $f$, since the characteristic of
$K$ is zero. This makes $f\circ a'$ nonzero in $S$, a contradiction.

The second assertion is derived dually in exactly the same way.
\end{proof}

That the docle $\doc(I) = M$ means that $M$ consists exactly of all the
maximal elements of $G\setminus U(I)$ with respect to divisibility. Since
$I$ is zero dimensional, then $U(I)$ is cofinite, so $G\setminus U(I)$ is a 
finite set of monomials, and hence each monomial therein is less than or equal to some
maximal monomial in $M$. Therefore $G\setminus U(I) = D(M)$.
Since $U(I)$ and $D(M)$ are disjoint we have the following.
\begin{claim}
\label{clm:UD}
For a zero dimensional monomial ideal $I$ with docle $\doc(I) = M$ 
we have a partition $G = U(I) \cup D(M)$.
\end{claim}
By Lemma~\ref{lmm:ann-mon=mon}, $J = \ann(M')$ is a monomial ideal of $R$.
Suppose $a\in J$ is a monomial, and so $a\circ a' = 0$ whenever $a'\in M'$. 
By (\ref{eqn:xoy}), this means precisely that $a$ does not divide
any of the monomials in $M$, and so by Claim~\ref{clm:UD},
$a\in G\setminus D(M) = U(I)\subseteq I$ since $I$ is spanned by
$U(I)$ as a $K$-vector space. This shows that $J\subseteq I$
and so we have that $J = I$. 
Although we suspect this to be well-known, we have derived it
in an elementary and a combinatorial way. This constitutes our
first description of the inverse ideal $J$ of a given set
of monomials $M$ as an ideal $I$
whose monomials in its socle that are not contained in $I$ are exactly
those in $M$. We summarize in the following.

\begin{observation}
\label{obs:I=J}
For any non-empty set $M$ of non-comparable monomials  of
$R = K[x_1,\ldots,x_d]$, the inverse ideal ${M'}^{-1} = \ann(M')$ of the 
set $M'\subseteq S$ of monomials corresponding to $M$ is identical to 
the unique zero dimensional monomial ideal
$I$ of $R$ with $\doc(I) = M$ from Theorem~\ref{thm:eye-catcher}.
\end{observation}

\begin{example}
\label{exa:d=2-2}
Consider again the set $M = \{x^2y^3,x^5y\}$ of non-comparable monomials of
$R = K[x,y]$ from Example~\ref{exa:d=2}. Since we are now dealing with the two
polynomial rings $R = K[x_1,x_2]$ and $S = K[y_1,y_2]$ we will here use the
variables $x_1$ and $x_2$ for $x$ and $y$ in $R$ respectively. In this case
we then have $M = \{x_1^2x_2^3, x_1^5x_2\}$ and so
$M' = \{y_1^2y_2^3, y_1^5y_2\}\subseteq S$. The unique zero-dimensional monomial
ideal $I$ of $R$ with $\doc(I) = M$ is here (by the above example) given by
$I = (x_2^4, x_1^3x_2^2,x_1^6)\subseteq R = K[x_1,x_2]$.

It is clear that
$J = \ann(M') = \{f\in R : f\circ(y_1^2y_2^3) = f\circ(y_1^5y_2) = 0\}$ is an ideal of
$S$. By Lemma~\ref{lmm:ann-mon=mon} we have that $J$ is a monomial ideal of $S$ and so
it suffices to describe the monomial generators of $J$. For a monomial
$a = x_1^ix_2^j\in R$ we then obtain
\[
0 = a\circ(y_1^2y_2^3) =
\left(\frac{\partial}{\partial y_1}\right)^i
\left(\frac{\partial}{\partial y_2}\right)^jy_1^2y_2^3, 
\]
which happens if and only if $i\geq 3$ or $j\geq 4$, or
$a \in (x_1^3,x_2^4)$. Likewise we obtain
\[
0 = a\circ(y_1^5y_2) =
\left(\frac{\partial}{\partial y_1}\right)^i
\left(\frac{\partial}{\partial y_2}\right)^jy_1^5y_2, 
\]
which happens if and only if $i\geq 6$ or $j\geq 2$, or
$a\in (x_1^6,x_2^2)$. Therefore,
\[
J = \ann(M') = (x_1^3,x_2^4)\cap (x_1^6,x_2^2) = (x_1^6,x_1^3x_2^2,x_2^4),
\]
which corresponds to (or actually is equal to, by relabeling the variables)
$I$ provided in the previous example and is consistent
with the above Observation~\ref{obs:I=J} as ${M'}^{-1} = \ann(M') = J$.
\end{example}

Next we describe combinatorially the inverse system of a monomial ideal,
in a dual fashion, in terms of the docle. 

Suppose we start with a monomial ideal $I$ of $R$, then 
by Definition~\ref{def:IS} the inverse system of $I$ can be given by
$I^{-1}= \{g'\in S : a\circ g' = 0 \mbox{ for each monomial } a\in I\}$.
By Lemma~\ref{lmm:ann-mon=mon} $I^{-1}$ is an $R$ module generated by monomials
in $S$ each of which by (\ref{eqn:xoy}) is not divisible by any monomial in 
$I'$, the monomial ideal in $S$ corresponding to $I$. In other words the inverse
system 
$I^{-1}$ is generated by all monomials in the downset $G'\setminus U(I')$ as
an $R$ module. This was also observed in \cite{Stolz} 
by different arguments. 

Note that if $I$ is a given monomial ideal of $R$, then the socle $\soc(I)$,
and hence the docle $\doc(I)$ is then uniquely determined by the downset
$D(\doc(I))$ in $G$ as $\doc(I)$ are the maximal elements of $D(\doc(I))$.
Also note that if $I$ is zero dimensional, so $U(I)$ is cofinite --
that is, the downset $G\setminus U(I)$ of $G$
is finite and therefore generated by its maximal elements $\doc(I)$ --  then 
$G\setminus U(I) = D(\doc(I))$, the downset generated by $\doc(I)$.
We summarize in the following, where we can describe the inverse system of a
monomial ideal in terms of its socle in the zero-dimensional case. 
\begin{corollary}
\label{cor:summary}
Let $I$ be a monomial ideal of $R$. The inverse system $I^{-1}$ 
is the $R$-submodule of $S$ generated by all monomials in the downset
$G'\setminus U(I')$. 

If $I$ is zero dimensional, then $G'\setminus U(I') = D(\doc(I'))$
and so the inverse system $I^{-1}$ is an $R$ module in $S$ 
minimally generated by $\doc(I')$.
\end{corollary}

We now describe the ``double inverse" operator for sets of monomials 
and monomial ideals and then a related closure operator.
Recall that a {\em closure operator} on a set ${\cal{S}}$ is a map
$\mathbb{P}(S)\rightarrow \mathbb{P}(S)$, from the power set of $S$ to itself,
given by $X\mapsto\overline{X}$ that is
(i) extensive: $X\subseteq\overline{X}$,
(ii) increasing: $X\subseteq Y\Rightarrow\overline{X}\subseteq\overline{Y}$
and
(iii) idempotent: $\overline{\overline{X}} = \overline{X}$.
First recall that $\ann(X) = \ann(I(X))$ 
for any subset $X\subseteq R$ of monomials, and likewise $\ann(Y') = \ann(R(Y'))$
for any subset $Y'\subset S$ of monomials. Adhering to 
Convention~\ref{conv:apostrophe},
the following is easily obtained.
\begin{proposition}
\label{prp:X-inverse}
For any subset $X\subseteq R$ of monomials we have
\[
X^{-1} = R\circ(G'\setminus U(X')) = R(G'\setminus U(X')),
\]
the $R$-module generated by the monomials in $G'\setminus U(X')$.

For any subset $Y'\subseteq S$ of monomials we have 
\[
{Y'}^{-1} = I(G\setminus D(Y)),
\]
the ideal of $R$ (that is the $R$-module) generated by the monomials
in $G\setminus D(Y)$.
\end{proposition}

\begin{rmk*}
For any set $Y'$ of monomials in the above proposition,
the ideal ${Y'}^{-1}$ is always a zero-dimensional ideal of $R$.
\end{rmk*}

With the above Proposition~\ref{prp:X-inverse} in mind we firstly
obtain purely combinatorially that for any subset $X\subseteq R$ of monomials that
\begin{eqnarray*}
(X^{-1})^{-1} & = & (R(G'\setminus U(X')))^{-1} \\
  & = &  (G'\setminus U(X'))^{-1} \\
  & = &  I(G \setminus D(G\setminus U(X))) \\
  & = &  I(G \setminus (G\setminus U(X))) \\
  & = &  I(U(X)) \\
  & = &  I(X).
\end{eqnarray*}
Likewise for any subset $Y'\subseteq S$ of monomials we also
obtain combinatorially by Proposition~\ref{prp:X-inverse} that
\begin{eqnarray*}
({Y'}^{-1})^{-1} & = & (I(G\setminus D(Y)))^{-1} \\
  & = & (G\setminus D(Y))^{-1} \\  
  & = & R(G'\setminus U(G\setminus D(Y))') \\  
  & = & R(G'\setminus U(G'\setminus D(Y'))) \\  
  & = & R(G'\setminus (G'\setminus D(Y'))) \\  
  & = & R(D(Y')) \\  
  & = & R\circ D(Y').
\end{eqnarray*}
Since the $R$-submodule of $S$ generated by $D(Y')$ is the same
as the one generated by $Y'$, we have the following summary:
\begin{corollary}
\label{cor:X-inv-inv}
For any subset $X\subseteq R$ of monomials we have
$(X^{-1})^{-1} = I(X)$, the ideal of $R$ generated by $X$.

For any subset $Y'\subseteq S$ of monomials we have
$({Y'}^{-1})^{-1}  = R(D(Y')) = R(Y')$, the
$R$-submodule of $S$ generated by $Y'$.
\end{corollary}

\begin{rmk*}
The first statement of the above Corollary~\ref{cor:X-inv-inv} 
can be viewed as a combinatorial version of a descriptive algebraic property of an 
artinian local rings $(R,\m)$, where here $\m$ is the unique maximal
ideal of a the ring $R$. Namely, that $R$ is a Gorenstein ring is equivalent
to $\ann(\ann(I)) = I$ for any ideal $I$ of $R$~\cite[Exercise 3.2.15]{Bruns-Herzog}.
However, since we here are only considering monomial ideals, purely from a combinatorial 
perspective, we do not need to restrict ourselves to the poset being cofinite; the 
equivalent of an ideal being artinian (or zero-dimensional). 

Note that both double-inverse operators 
the map $X\mapsto \overline{X} := (X^{-1})^{-1}$ for 
sets $X$ of $R$ on one hand and the map $Y' \mapsto \overline{Y'} := ({Y'}^{-1})^{-1}$ 
for sets $Y'\subseteq S$ on the other, are clearly closure operators 
by Corollary~\ref{cor:X-inv-inv}.  
\end{rmk*}

\section{Closure operators on upsets and monomial ideals}
\label{sec:closure-op}

There is a more interesting closure operator than the ones mentioned 
in the last remark in Section~\ref{sec:socle-IS}
(though not relating to the usual partial order)
on monomial ideals that we
will now discuss. Since much of what we present can
be stated more generally in terms of lattices and posets, we will derive most of
what follows in terms of posets and then state what it means for monomial
ideals of $R$.

For a poset $P$ and an upset $U\subseteq P$ let 
\begin{equation}
\label{eqn:closure(U)} 
\overline{U} := P\setminus D(\doc(U)).
\end{equation}
First note that $\overline{U}$ is an upset of $P$. 
By Claim~\ref{clm:M=doc} we have that $M = \doc(U)$ is an antichain and
that $U\subseteq P\setminus D(\doc(U)) = \overline{U}$.

Next consider the double operation 
$\overline{\overline{U}} = P\setminus D(\doc(\overline{U}))$.
By definition of $\doc $ and since $\doc(U)$ is an antichain, we then obtain 
\begin{equation}
\label{eqn:doc-overline}    
\doc(\overline{U})
= \doc(P\setminus D(\doc(U))) 
= \max(P\setminus (P\setminus D(\doc(U))))
= \max(D(\doc(U)))
= \doc(U),
\end{equation}
and so we obtain
\[
\overline{\overline{U}} 
= P\setminus D(\doc(\overline{U}))
= P\setminus D(\doc(U)) = \overline{U}.
\]
We summarize this in the following:
\begin{proposition}
\label{prp:2/3-closure}
If $P$ is a poset, $U\subseteq P$ is an upset and
$\overline{U}$ is as in (\ref{eqn:closure(U)}), then 
we have the extensive and idempotent properties:
(i) $U\subseteq\overline{U}$ and
(ii) $\overline{\overline{U}} = \overline{U}$,
and so two of the three conditions for a closure operator 
are satisfied.
\end{proposition}

\begin{example} 
Consider the posets $P = (\nats_0^2,\leq)$ and the upsets 
$U$ and $V$ generated by the sets $\{(1,1)\}$ and $\{ (1,0),(0,1)\}$ 
respectively. Clearly we have here that $U\subset V$, 
$\doc(U) = \emptyset$ and
$\doc(V) = \{(0,0)\}$, and hence $\overline{U} = \nats_0^2$ whereas 
$\overline{V} = V \neq \nats_0^2$. 
\end{example}

We see from the above example that the operation $U\mapsto \overline{U}$ 
does not in general satisfy the monotone property 
$U\subseteq V \Rightarrow \overline{U}\subseteq \overline{V}$.

However, instead of defining a partial order on upsets of a given poset $P$ by the 
usual set theoretic inclusion, we consider another partial order 
on the set ${\cal{U}}(P)$, the set of all upsets of the poset $P$.
This will be the main topic for the rest of this section. 

For upsets $U,V\in {\cal{U}}(P)$ consider a binary relation $\sqsubseteq$ 
defined by 
\begin{equation}
\label{eqn:sqsubseteq}
U\sqsubseteq V \Leftrightarrow \doc(U)\supseteq \doc(V) \mbox{ and } U\subseteq V.
\end{equation}
The following is straightforward to verify.
\begin{observation}
\label{obs:sq=po}
If $\sqsubseteq$ is the binary relation from (\ref{eqn:sqsubseteq}),
then $({\cal{U}}(P),\sqsubseteq)$ is a poset.
\end{observation}

\begin{example}
\label{exa:UVW}
Consider again the posets $P = (\nats_0^2,\leq)$ and the  $U$, $V$ and $W$
given by the following:
\begin{eqnarray*}
U & := & U(((2,3),(3,2),(4,1)), \\
V & := & U((2,3),(3,1)), \\
W & := & U((1,4),(2,3),(3,2),(4,1)).
\end{eqnarray*}
In this case we obtain:
\begin{eqnarray*}
 \doc(U) & = & \{(2,2),(3,1)\}, \\
 \doc(V) & = & \{(2,2)\},\\
 \doc(W) & = &\{(1,3),(2,2),(3,1)\},
\end{eqnarray*}
and therefore we have by inspection that $U\subset V$, $U\subset W$,
$\doc(V)\subset \doc(U)$ and $\doc(U)\subset \doc(W)$; all as strict 
containment. Hence, by definition (\ref{eqn:sqsubseteq}), we have 
$U\sqsubseteq V$ and the upsets $U$ and $W$ are incomparable 
w.r.t.~$\sqsubseteq$.

In addition we have
\begin{eqnarray*}
\overline(U) & = & U((0,3),(3,2),(4,0)), \\
\overline(V) & = & U((0,3),(3,0)), \\
\overline(W) & = & U((0,4),(2,3),(3,2), (4,0)).
\end{eqnarray*}
\end{example}

By (\ref{eqn:doc-overline}) we have $\doc(\overline{U}) = \doc(U)$ and
by Proposition~\ref{prp:2/3-closure} we have $U\subseteq\overline{U}$
and so we have extensivity $U\sqsubseteq\overline{U}$. By the transitivity 
of the usual subset relation $\subseteq$ and again by (\ref{eqn:doc-overline}), 
we have the monotone property for $\sqsubseteq$ as well and by
Proposition~\ref{prp:2/3-closure} we have the following.
\begin{theorem}
\label{thm:sqsubseteq-closure-op}
If $P$ is a poset, ${\cal{U}}(P)$ is the set of all upsets of $P$,
$\sqsubseteq$ is as in (\ref{eqn:sqsubseteq}) and $\overline{U}$ is as
given in (\ref{eqn:closure(U)}), then $U\mapsto\overline{U}$ is a 
closure operator on the poset $({\cal{U}}(P),\sqsubseteq)$.
\end{theorem}
This can now be applied to monomial ideals of $R = K[x_1,\ldots,x_d]$.
Consider a monomial ideal $I$ of $R$. The elements of $M = \doc(I)$
always form an antichain in the monoid $G$. From $I$ we can form
the ideal $\overline{I} := I(G\setminus D(\doc(I)))$, the ideal generated
by $G\setminus D(\doc(I))$. Since each monomial ideal in $R$ is uniquely determined
by its set of monomial from $G$, the relation $\sqsubseteq$ from 
(\ref{eqn:sqsubseteq}) clearly carries over
to monomial ideals as well and is given by the same conditions as
in (\ref{eqn:sqsubseteq}), that is,
$I\sqsubseteq J \iff \doc(I)\supseteq \doc(J) \mbox{ and } I\subseteq J$,
for any monomial ideals $I$ and $J$ of $R$. Therefore we have the following
corollary from Theorem~\ref{thm:sqsubseteq-closure-op}.
\begin{corollary}
\label{cor:closure}
If $\cal{I}(R)$ is the set of all monomials ideal of $R$, then the map
${\cal{I}}(R)\rightarrow {\cal{I}}(R)$ given by 
$I\mapsto\overline{I} = I(G\setminus D(\doc(I)))$ is a closure operator
on the poset $(\cal{I}(R),\sqsubseteq)$. 
\end{corollary}
Since $R$ is Noetherian, every monomial ideal $I$ is generated
by finitely many monomials, say $n$ of them. Using the characterization of the docle $\doc(I)$
right before Definition~\ref{def:doc(I)}, as $\doc(I) = \soc(I)\setminus I$, 
each point $a\in \doc(I)$ is uniquely determined by 
exactly $d$ generators of $I$ (see Introduction and Section 2 of~\cite{A-Epstein})
and hence $|\doc(I)|$ is trivially bounded above by $\binom{n}{d}$, where
$n$ is minimal number of generators for the monomial ideal $I$, and is, 
in particular, a finite set. In fact, a more careful analysis of the 
maximum number $c_d(n)$ of $\doc(I)$ for a monomial ideal 
$I\subseteq R = K[x_1,\ldots,x_d]$ with at most $n$ generators
shows that $c_d(n) = \Theta(n ^{\lfloor d/2\rfloor})$ for large 
$n$~\cite{A-outside}. Hence, if $I$ is a monomial ideal of $R$, then
$\doc(I)$ is a finite subset of $G\subseteq R$ and therefore so
is $D(\doc(I))\subseteq G$. The set $G\setminus D(\doc(I))$ is therefore
a cofinite subset of $G$ and so the closure 
$\overline{I} = I(G\setminus D(\doc(I)))$ of $I$ is a zero-dimensional 
ideal of $R$. 

Conversely, if $I$ is a zero-dimensional ideal of $R$, then $G\setminus I$ 
is a finite set and so every element of $G\setminus I$ is a less than
or equal to some maximal element of $G\setminus I$, that is
$G\setminus I = D(\doc(I))$. That is to say, we have equality in
the subset relation (ii) $D(M)\subseteq P\setminus U$ in Claim~\ref{clm:M=doc}
for the poset $P = G$. Since then $G = I\cup D(\doc(I))$ is a partition 
of $G$ into an upset and downset, we get that
\[
I = G\setminus D(\doc(I)) = I(G\setminus D(\doc(I))) = \overline{I}
\]
is a closed set in the poset $(\cal{I}(R),\sqsubseteq)$. Hence we have 
the following.
\begin{corollary}
\label{cor:zero-dim=closed}
In the poset $(\cal{I}(R),\sqsubseteq)$ from Corollary~\ref{cor:closure},
a monomial ideal $I$ is closed if and only if it is zero-dimensional.
\end{corollary}

We conclude this section with a simple yet motivating example that
demonstrates that the assumption that $I$ is a monomial ideal is vital
for our combinatorial correspondence between the ideal and its inverse system.

\begin{example}
\label{exa:motivating}
  Consider the homogeneous ideal $I = (x^3, y^2 - xy)$ of $R = K[x,y]$,
the polynomial ring
over a given field $K$. Note that $R/I$ is
(i) local with maximal ideal $\overline{\m}$ where $\m = (x,y)$,
(ii) a finite dimensional graded $K$-algebra $R/I = \bigoplus_{i=0}^3(R/I)_i$ 
(of dimension $6$), and
(iii) $R/I$ is Gorenstein with its socle being simple and generated by
$\overline{x}^2\overline{y}$ both as a $K$-vector space and as an $R/I$
module.

For the set $M = \{x^2y\}\subseteq K[x,y]$, we know by
Observation~\ref{obs:I=J} that the unique zero dimensional monomial
ideal $J = (x^3,y^2)\subseteq K[x,y]$ with $\doc(J) = M$ is identical
to the ideal $\ann(M')$, and so $J = \{f\in K[x,y] : f\circ(x^2y) = 0\}$
where the $f$ acts as a differential operator by (\ref{eqn:xoy}).

That $\doc(J) = \{x^2y\}$ (modulo $J$)
means that the socle of $R/J$ is generated
by $\overline{x}^2\overline{y}$, as is the case with $R/I$ from above.
However, when the ideal $I$ acts as a differential operator, we note that
$(y^2 - xy)\circ(x^2y) = -2x \neq 0$, and so we have $I \neq\ann(M')$ here.
There is, however, a known way to describe homogeneous, zero-dimensional 
ideals in local Gorenstein rings in terms of differential operators.
We will discuss this in Section~\ref{sec:Gor-inv}.
Before that, however, we will in the next section discuss some algebraic
properties of ideals that have combinatorial implications.
\end{example}

\section{The docle from an algebraic viewpoint}
\label{sec:alg-observations}

Consider the monoid $\Gamma = \mathbb{Z}$ or $\mathbb{Z}^d$ for 
some $d\in \mathbb N$.  
Recall that a \emph{graded local} (or  \emph{gr-local}) ring is a $\Gamma$-graded 
ring $R$ such that the ideal $\m$ generated by the homogeneous nonunits of $R$ 
satisfies $\m \neq R$.  We also write ``$(R,\m)$ is gr-local'' in this case.

Recall that if $I$ is a proper homogeneous ideal in a gr-local ring $(R,\m)$, 
then $(R/I, \m/I)$ is also gr-local.

Key examples: 
\begin{itemize}
    \item $R=K[x_1, \ldots, x_d]$, with the usual $\mathbb{Z}$-grading (i.e. $R_j$ 
    is the $K$-vector space generated by all forms of degree $j$), is gr-local with 
    $\m = (x_1, \ldots, x_d)$.  The homogeneous ideals with respect to this gradings 
    are the ideals generated by homogeneous polynomials (i.e. $j$-forms for various 
    $j \in \mathbb{N}_0$).  This includes the monomial ideals but also many others.
    \item Let $R$ be as above, but assign the $\mathbb{Z}^d$-grading such that the 
    graded component of degree $(a_1, \ldots, a_n)$, where each $a_i \in \mathbb{N}_0$, 
    is just the 1-dimensional $K$-vector space generated by the monomial 
    $x_1^{a_1} \cdots x_d^{a_d}$, then again $R$ is gr-local with respect to $\mathbb{Z}^d$, 
    with $\m = (x_1, \ldots, x_d)$. The homogeneous ideals with respect to this grading 
    are precisely the monomial ideals.
    \item If $I$ is a monomial ideal in the ring $R$ above, then it is homogeneous 
    in both of the above 
    gradings, so $(R/I, \m/I)$ is also gr-local in both gradings.
\end{itemize}

Recall~\cite[Exercise 3.5]{Eisenbud} that if $I$ is a homogeneous ideal in a 
$\Gamma$-graded Noetherian ring, where $\Gamma=\mathbb Z$ or $\mathbb{Z}^n$, 
that all its associated primes are homogeneous with respect the $\Gamma$-grading, 
and there is an irredundant (or reduced) primary decomposition $I = \q_1 \cap \ldots \cap \q_h$ 
where each $\q_i$ is homogeneous with respect to $\Gamma$.
\begin{proposition}
\label{prp:idealanddoc}
Let $R=K[x_1, \ldots, x_d]$ and $\m = (x_1, \ldots, x_d)$.  
Let $I \subseteq J$ be monomial ideals in $R$.  Then there is a natural 
map $u: (I:\m)/I \rightarrow (J:\m)/J$, such that 
\begin{enumerate}
[label=\emph{(\alph*)}, ref=(\alph*)]
    \item\label{it:idealdocinj} $\doc(I) \subseteq \doc(J) \iff u$ is injective, and
    \item\label{it:idealdocsurj} $\doc(J) \subseteq \doc(I) \iff u$ is surjective.
\end{enumerate}
\end{proposition}
\begin{proof}
First, let us construct the map $u$.  Since $I \subseteq J$, we have 
$(I:\m) \subseteq (J:\m)$.  Let $v: (I:\m) \ra (J:\m)/J$ be the composition of the 
inclusion map $(I:\m) \into (J:\m)$ with the natural surjection $(J:\m)/J$.  
Since $I \subseteq J$, we have $v(I) \subseteq I+J/J=0$, so $I \subseteq \ker(v)$, 
so we get a map $u: (I:\m)/I \ra (J:\m)/J$ of $R$-modules induced by the 
First Isomorphism Theorem.  It is given by $u(f+I) = f+J$.

Note that for any monomial ideal $H$, there is a one-to-one correspondence between  
$\doc(H)$ and the monomials in $(H:\m)$ that are not in $H$, as described earlier
right before Definition~\ref{def:doc(I)}.

Next, we prove \ref{it:idealdocinj}.  Suppose $\doc(I) \subseteq \doc(J)$.  
Since $u$ is a map of multigraded modules, to prove injectivity it is enough to show that 
$\ker(u)$ contains no nonzero multigraded elements. Accordingly, let $t \in (I:\m)/I$ 
be a nonzero multigraded element.  Then $t$ is represented by a monomial $s$ that is 
in $(I:\m) \setminus I$.  But then by the above paragraph, $s \in \doc(I)$, hence 
$s\in \doc(J)$ by assumption.  Thus, $s \notin J$, so $0 \neq s+J = u(t) \in (J:\m)/J$.

For the converse, suppose $u$ is injective.  Let $s \in \doc(I)$.  Then 
$0 \neq s+I \in (I:\m)/I$, so by assumption, $0 \neq u(s+I) = s+J \in (J:\m)/J$.  
Thus, $s \in \doc(J)$.

Lastly, we prove \ref{it:idealdocsurj}.  Suppose $\doc(J) \subseteq \doc(I)$.  
Since $u$ is a map of multigraded modules, to prove surjectivity it is enough to 
show that every nonzero multigraded element of $(J:\m)/J$ is 
in the image of $u$.  So let $t$ be such an element.  Then $t = s+J$, where 
$s\in \doc(J)$.  But then $s\in \doc(I)$ by assumption, so $0 \neq s+I \in (I:\m)/I$.  
Then we have $u(s+I) = s+J = t$, so $u$ is surjective.

Conversely, suppose $u$ is surjective.  Let $s \in \doc(J)$.  Then $0 \neq s+J \in (J:\m)/J$, 
so since $u$ is surjective, we have $s+J \in \im(u)$. Since $u$ is a multigraded map, 
$s+J = u(t)$ for some multigraded $t\in (I:\m)/I$.  But by definition of $u$, 
it follows that $t=s+I$.  Hence, $0 \neq s+I \in (I:\m)/I$, whence $s\in \doc(I)$.
\end{proof}

Next, consider the following interesting test for equality of artinian (or zero-dimensional) ideals.

\begin{proposition}\label{pr:equalifinj}
Let $(R,\m)$ be a local (resp. gr-local) Noetherian ring.  Let $I \subseteq J$ be proper (resp. proper homogeneous) ideals such that $R/I$ is (gr-)artinian.  Then $J=I$ if and only if the natural map $u: \frac{I:\m}I \ra \frac{J:\m}J$ is injective.
\end{proposition}

\begin{proof}
The kernel of $u$ is given by $\frac{(I:\m) \cap J}I$, 
which can be identified with $(0 :_{J/I} \m)$.  Thus, $(0:_{J/I}\m) = 0$.  But $N :=J/I$ is a 
finite length (homogeneous) $R$-module, so there is some nonnegative integer $h$ with $\m^h N =0$.  
Let $h$ be minimal and assume $h \geq 1$.  Then $\m^{h-1}N \neq 0$, so there is some 
$0 \neq z \in \m^{h-1}N$.  We have $\m z \subseteq \m^hN=0$, so $0 \neq z \in (0:_N\m) = 0$, 
a contradiction.  Hence, $h=0$, so $J/I = N = RN = \m^0 N = \m^h N = 0$.  
Thus, $J=I$.
\end{proof}

\begin{example}
Upon reading Proposition~\ref{pr:equalifinj}, it is natural to ask the question: if the map $u$ is \emph{surjective}, is this also enough to imply that $J=I$?  The answer is ``no''.  To see this, let $R = k[X,Y]$, $\m = (X,Y)$, $I = \m^3$, and $J = (X^2, Y^2)$. Then $(I:\m) = \m^2 = (X^2, XY, Y^2)$ and $(J:\m) = (XY)$.  Then the map $u$ is surjective, as the image of $XY$ in $(I:\m)/I$ maps to the generator $\overline{XY}$ of the cyclic module $(J:\m)/J$. But it is not injective (as the images of $X^2$ and $Y^2$ in $(I:\m)/I$ map to zero in $(J:\m)/J$ since $X^2, Y^2 \in J$), and of course $I \neq J$.
\end{example}

Combining Propositions~\ref{prp:idealanddoc} and \ref{pr:equalifinj}, we obtain the following:
\begin{corollary}
\label{cor:poset-idealanddoc}
Let $R=K[x_1, \ldots, x_d]$ and $\m = (x_1, \ldots, x_d)$.  Let $I \subseteq J$ be artinian 
ideals with $\doc(I) \subseteq \doc(J)$.  Then $I=J$.
\end{corollary}
\begin{proof}
By Proposition~\ref{prp:idealanddoc} the induced map  $u: (I:\m) /I \ra (J:\m)/J$ is injective. Then by Proposition~\ref{pr:equalifinj}, $I=J$.
\end{proof}
We then get the following combinatorial result, from algebraic methods.
\begin{corollary}
  If $P = \mathbb{N}_0^d$ is a poset and $U,V \subseteq P$ are cofinite upsets
  satisfying $U \subseteq V$ and $\doc(U) \subseteq \doc(V)$, 
then $U=V$.
\end{corollary}
Our next goal is the following theorem, which amounts to an algebraic generalization 
of what we know combinatorially for monomial ideals with common socles.
Recall that
\[
(I:\m^{\infty}) = \bigcup_{n\in\nats}(I:\m^n).
\]
\begin{theorem}
\label{thm:expandtomprimary}
Let $(R,\m)$ be a local  (resp. gr-local) Noetherian ring.  Let $I \subseteq R$ be an 
ideal (resp. a homogeneous ideal) with $(I :\m) \neq I$.  Then there exist ideals 
(resp. homogeneous ideals) $J,H$ such that $H$ is $\m$-primary, $I=J \cap H$, $(J:\m)=J$, 
and the natural map $u: \frac{(I:\m)}I \ra \frac{(H:\m)}H$ is an isomorphism.  
Moreover, there is only one such choice of $J$; namely, $J=(I:\m^\infty)$, and the set of ideals $H$ that fulfill the conditions form an antichain.
\end{theorem}

We start with the surely well-known preparatory lemma, whose proof we include for 
the convenience of the reader:

\begin{lemma}\label{lem:Nakart}
Let $(R,\m)$ be a local (resp. gr-local) ring and $M$ a (homogeneous) $R$-module such that for any (homogeneous) element $x\in M$, some power of $\m$ annihilates $x$.  Suppose $(0:_M \m) = 0$.  Then $M=0$.
\end{lemma}

\begin{proof}
Choose any (homogeneous) element $x\in M$.  Let $h \in {\mathbb N}_0$ be minimal such that $\m^h x = 0$.  Suppose $h\geq 1$.  Then $\m^{h-1}x\neq 0$ (and is a graded submodule of $M$).  Thus, we can choose a nonzero (homogeneous) element $y$ of $\m^{h-1}x$.  Then $\m y = 0$, so $y \in (0:_M \m) = 0$, a contradiction.  Thus, $h=0$, so that $x\in Rx = \m^0x = \m^h x = 0$.  Thus, $M=0$.
\end{proof}

Next we prove the following lemma showing us how to increase our ideals.
\begin{lemma}
\label{lem:expandmpcomp}
Let $(R,\m)$ be a local (resp. gr-local) ring and let $J,H$ be proper ideals 
(resp. proper homogeneous ideals). Set $I := J\cap H$.  If $H+(I:\m) \neq (H:\m)$, 
then for any (homogeneous) element $t \in (H:\m) \setminus (H+(I:\m))$, we have 
$I = J \cap (H+(t))$.
\end{lemma}
\begin{proof}
Since $\m t \subseteq H$ but $t\notin H$, we have $\m \subseteq (H:t) \neq R$, so 
by (homogeneous) maximality of $\m$, we have $\m = (H:t)$.
Now, let $\alpha \in J \cap (H+(t))$.  (In the graded case, make sure to choose 
$\alpha$ to be a homogeneous element.)  Then $\alpha = y+ct$ for some (homogeneous) 
$y \in H$ and (homogeneous) $c\in R$.  For any $a\in \m$, we have $at \in H$, so 
$a\alpha = ay + cat \in J\cap H = I$.  Thus, $\alpha \in (I:\m)$.  But $t \notin H+(I:\m)$, 
and $ct = -y+\alpha \in H+(I:\m)$, so $c$ is not a unit.  Since $(S,\m)$ is local 
(resp. gr-local and $c$ is homogeneous), it follows that $c\in \m = (H:t)$.  
Thus, $ct \in (H:t)t \subseteq H$.  Therefore, $\alpha = y+ct \in J\cap HJ = I$.  
Since $\alpha \in J \cap (H+(t))$ was arbitrary (resp. was an arbitrary homogeneous element, 
and $J \cap (H+(t))$ is homogeneous), it follows that $J \cap (H+(t)) \subseteq I$.  
But $I \subseteq J \cap H \subseteq J \cap (H+(t))$.  Hence, $I = J\cap(H+(t))$.
\end{proof}

Next, we have the following result equating a number of different concepts.
\begin{proposition}\label{pr:MECs}
Let $(R,\m)$ be a local (resp. gr-local) ring.  Let $I$ be a (homogeneous) ideal with $I \neq (I:\m)$, and let $I = J \cap H$, where $J,H$ are (homogeneous) ideals such that $J=(J:\m)$ and $H$ is $\m$-primary.  Let $\mathcal{C} := \{G \subseteq R \mid G$ (homogeneous) $\m$-primary ideal with $I=J \cap G\}$.  Let $\varphi:  \frac{I:\m}I \ra \frac{H:\m}H$ be the induced map. The following are equivalent: 
\begin{enumerate}
    \item\label{it:Hmax} $H$ is a maximal element of $\mathcal{C}$.
    \item\label{it:phisurj} $\varphi$ is surjective.
    \item\label{it:phiiso}  $\varphi$ is an isomorphism.
    \item\label{it:formula} $(I:\m) + H = (H:\m)$.
\end{enumerate}
\end{proposition}

\begin{proof}
First note that $\varphi$ is injective under these hypotheses.  Indeed, let $x \in (I:\m)\cap H$.  Then since $I \subseteq J$ and $J = (J:\m)$, we have $x\in J$.  Since also $x\in H$, we have $x\in J \cap H = I$.  Thus, $(I:\m)\cap H = I$. thus, $\ker \varphi = \frac{(I:\m) \cap H}I = \frac II = 0$.  It follows that (\ref{it:phisurj}) $\iff$ (\ref{it:phiiso}).

(\ref{it:Hmax}) $\implies$ (\ref{it:formula}): We prove the contrapositive.  Suppose $(I:\m)+H \neq (H:\m)$.  Let $t\in (H:\m) \setminus (I:\m)+H$, and in the graded case make sure to choose $t$ homogeneous.  Then by Lemma~\ref{lem:expandmpcomp}, we have $J \cap (H+(t)) = I$ and clearly $H \subsetneq H+(t)$.  Moreover, $H+(t)$ is $\m$-primary.  This follows from the facts that ($H+(t)$ is homogeneous, that) $\m$ is the only (homogeneous) prime ideal to contain $H$, and $H  \subseteq H+(t) \subseteq \m$.  Hence $H+(t) \in \mathcal C$ and strictly contains $H$, so $H$ is a nonmaximal element of $\mathcal C$.

(\ref{it:formula}) $\implies$ (\ref{it:phisurj}): Let $\alpha \in \frac{H:\m}H$. Choose $u\in H:\m$ so that $\bar u = \alpha$.  Then $u=v+w$ for some $v\in (I:\m)$ and $w\in H$.  Then $\varphi(v+I) = v+H = v+w+H = u+H = \alpha$.

(\ref{it:phisurj}) $\implies$ (\ref{it:formula}): It is clear that $(I:\m)+H \subseteq (H:\m)$.  Conversely, let $x\in (H:\m)$.  Then $x+H \in \frac{H:\m}H =\im(\varphi)$.  That is, there is some $y \in (I:\m)$ with $\varphi(\bar y) = \bar x$.  That is, $y+H = x+H$, so $x\in y+H$.  Since $y \in (I:\m)$, it follows that $x \in y+H \subseteq (I:\m)+H$.

(\ref{it:formula}) $\implies$ (\ref{it:Hmax}): Let $G\in \mathcal{C}$ with $H \subseteq G$.  Then: 
\[
(H:_G\m) = (H:\m) \cap G = (H + (I:\m)) \cap G =  H + ((I:\m) \cap G)\subseteq H+(J \cap G) = H+I = H.
\]
Thus, $(0:_{G/H} \m) = \frac{(H:_G\m)}H = \frac HH = 0$.  But since $H$ is $\m$-primary, some power of $\m$ kills $R/H$, and thus also $G/H$.  Then by Lemma~\ref{lem:Nakart}, $G/H=0$, whence $G=H$.
\end{proof}

\begin{proof}[Proof of Theorem~\ref{thm:expandtomprimary}]
Since $(I:\m) \neq I$, there is some (homogeneous) $x \in (I:\m) \setminus I$.  
Then $\m \subseteq (I:x) \neq R$, so by 
maximality of $\m$ (among homogeneous ideals), we have $\m = (I:x)$.  Thus, $\m \in \Ass_R(R/I)$ is an associated
prime ideal of $R/I$.  Accordingly, there is some irredundant primary decomposition 
$I=\q_1 \cap \ldots \cap \q_h$ of $I$ such that the last one $\q_h$ is $\m$-primary (which can be chosen with all $\q_i$ homogeneous by \cite[Exercise 3.5]{Eisenbud}).  
Set $J := \q_1 \cap \cdots \cap \q_{h-1}$.

To see that $J=(J:\m)$, let $y\in (J:\m)$.  Then for each $1\leq i \leq h-1$, we have 
$\m y \subseteq \q_i$.  Since $\m \nsubseteq \sqrt{\q_i}$, choose 
$a\in \m \setminus \sqrt{\q_i}$.  Then since $ay \in \q_i$, $a \notin \sqrt{\q_i}$,
and $\q_i$ is primary, it follows that $y\in \q_i$.  Since this holds for all $i\leq h-1$, 
we have $y\in \bigcap_{i=1}^{h-1} \q_i=J$.

Next, let $\mathcal C := \{$(homogeneous) $\m$-primary ideals $L \mid I = J \cap L\}$.
We have $\q_h \in \mathcal C$, so ${\mathcal C} \neq \emptyset$.  Since $R$ is Noetherian, $\mathcal C$ 
admits a maximal element; call it $H$.  Then by Proposition~\ref{pr:MECs}, $u$ is an isomorphism.  On the other hand, by the same proposition, any ideal in $\mathcal C$ that satisfies the conditions will be maximal in $\mathcal C$, so they form an antichain.

For the final statement, recall that $J = (J:\m)$, as we showed earlier in the proof.
Conversely, let $J,H$ be arbitrary (homogeneous) ideals with $H$ $\m$-primary such 
that $I=J\cap H$ and $J = (J:\m)$.  There is some $h$ with $\m^h \subseteq H$.  
So if $x\in J$, then $\m^h x \subseteq J \cap H = I$, so that $x \in (I:\m^h) \subseteq (I:\m^\infty)$.  
Hence, $J \subseteq (I:\m^\infty)$. For the reverse containment, let $y \in (I:\m^\infty)$.  
Then there is some nonnegative integer $s$ such that $\m^s y \subseteq I \subseteq J$.  
Let $\ell \in \mathbb{N}_0$ be minimal with $\m^\ell y \subseteq J$ and suppose 
$\ell\geq 1$. Then $y \in (J:\m^\ell) = ((J:\m):\m^{\ell-1}) = (J:\m^{\ell-1})$, 
contradicting the minimality of $\ell$.  Thus $\ell=0$, so that $y\in J$.  
Hence, $(I:\m^\infty) \subseteq J$. 
\end{proof}

In the polynomial ring case, we then obtain the following result by the above theorem in conjunction with our earlier work in~\cite{A-Epstein}.

\begin{corollary}
\label{cor:uniqueidealdecomp}
Let $R = K[x_1, \ldots, x_d]$.  Let $I$ be a monomial ideal with a nonempty docle, so we have $\doc(I) \neq \emptyset$.  Then there is a unique pair of monomial ideals $J$, $H$ such that $I=J \cap H$, $H$ is zero-dimensional, $\doc(H) = \doc(I)$, and $\doc(J) = \emptyset$.
\end{corollary}

\begin{proof}
In the $\mathbb{Z}^n$-grading given by monomial exponents, ``homogeneous ideal'' means ``monomial ideal''.  Moreover, $\doc(J)=\emptyset \iff (J:\m)=J$ (essentially by definition), and $\doc(H) = \doc(I) \iff$ the natural map $\frac{I:\m}I \ra \frac{H:\m}H$ is an isomorphism (by Proposition~\ref{prp:idealanddoc}).  Then Theorem~\ref{thm:expandtomprimary} provides the pair $J,H$ as in the statement, with $J$ unique.  Moreover, $H$ is also unique by Theorem~\ref{thm:eye-catcher}.
\end{proof}

We obtain the following combinatorial corollary to the algebraic work above.
\begin{corollary}
\label{cor:uniqueupsetdecomp}
Let $P = \mathbb{N}_0^n$.  Let $U \subset P$ be an upset such that 
$\doc(U) \neq \emptyset$.  Then there is a unique representation of $U$ in the 
form $U=V \cap W$, where 
$V,W$ are upsets, $V$ is cofinite,
$\doc(W) = \emptyset$, and 
$\doc(U) = \doc(V)$.
\end{corollary}

Corollary~\ref{cor:uniqueidealdecomp} is particularly interesting, in that its analogue in the local case is \emph{false}.  Indeed, recall the following result, with notation changed to match ours here.

\begin{theorem}[{\cite[Theorem 2.8]{HRS-emb1}}]\label{thm:HRS}
Let $(R,\m)$ be a \emph{local} Noetherian ring. Let $I$ be a non-$\m$-primary ideal of $R$ such that $(I:\m)\neq I$.  Let $J$ be as in Theorem~\ref{thm:expandtomprimary} and $\mathcal C$ as in Proposition~\ref{pr:MECs}.  Then $I$ is the intersection of all the maximal elements of $\mathcal C$.
\end{theorem}

\begin{example}
It is instructive to examine the case of the ``Emmy ideal'' $I=(x^2,xy)$, either in $K[\![x,y]\!]$ (as in the article \cite{HRS-emb1}) or in $K[x,y]$. In the former (i.e. no grading) case, the maximal elements of $\mathcal C$, i.e. the ``Maximal Embedded Components'', or MECs, of $I$ are fully analyzed in \cite[Section 3]{HRS-emb1}.  They are: $(x^2, y)$, and all ideals of the form $(x^2, xy, \lambda x + y^n)$ where $0\neq \lambda \in K$ and $n \geq 1$.  One sees quickly that intersecting all such MECs will yield $I$.

Now consider the $\mathbb Z$-graded case.  That is, $R=K[x,y]$, $I=(x^2, xy)$, and $\mathcal C$ consists of all ideals $H$ that appear in an intersection such as in Proposition~\ref{pr:MECs} and are homogeneous in the ordinary sense (i.e. may or may not be monomial ideals). Then the maximal elements of $\mathcal C$ are $(x^2, y)$ and also the ideals $(x^2, xy, \lambda x + y)$, where $0\neq \lambda \in K$.  The intersection of all these is $(x^2, xy, y^2)$.  So this case is really intermediate, in that there are lots of maximal elements, but they do not intersect down to $I$.

Finally, consider the $\mathbb Z^n = \mathbb Z^2$-graded case, where $\mathcal C$ consists of the \emph{monomial} ideals $H$ that appear in an intersection such as in Proposition~\ref{pr:MECs}. Then as shown above, there is a \emph{unique} maximal element of $\mathcal C$, which happens to be $(x^2,y)$.  Indeed, the intersection corresponding to Corollary~\ref{cor:uniqueidealdecomp} is $(x^2,xy) = (x) \cap (x^2, y)$.
\end{example}

\section{Homogeneous ideals in Gorenstein rings in terms of inverse systems}
\label{sec:Gor-inv}

The main purpose of this section is to establish the 
"if and only if" part in Corollary~\ref{cor:ann-iff-monomial}.
Namely, we demonstrate by explicit computations 
that in the case of an artinian Gorenstein ideal, the ideal is the inverse ideal 
of the sole monomial in its socle if and only if the ideal itself is monomial.

Most of what is discussed in this section is well-known in
a more abstract and general form in the literature,
\cite{Eisenbud},
\cite{Eisenbud-Buchsbaum},
\cite{CoxLittleSchenck} and 
\cite{Bruns-Herzog}, 
but to the best of our knowledge the more explicit, specific statements 
we deduce here have not appeared in the literature
before. It therefore seems worthwhile to write down and prove a few explicit 
observations and propositions for our deductions to come. In particular, we will 
carry out computations relating to the specific term order (or {\em monomial order}) LEX 
on the monomials in the monoid $[x_1,\ldots,x_d]$, which is a linear ordering that respects 
multiplication (see the definition immediately after Observation~\ref{obs:soc-monomial}).

As described in~\cite[p.~376]{Eisenbud}, 
a homogeneous, zero-dimensional
Gorenstein ideal of the polynomial ring $R = K[x_1,\ldots,x_d]$ has the
form of a colon ideal (also knows as an ideal quotient) $I = ((x_1^k,\ldots,x_d^k):p)$ 
where $p\in R$ is a homogeneous polynomial. The polynomial $p$ is sometimes called 
the {\em dual socle generator} of $R/I$. 
As $I$ contains $(x_1^k,\ldots,x_d^k)$, the only maximal ideal of $R$
containing $I$ is $\m = (x_1,\ldots,x_d)$.
Since $I$ is homogeneous, it follows that $R/I$ is a graded finite dimensional $K$-algebra.

We will now describe an explicit basis for $R/I$ as a vector space over $K$
and some computational properties that we will use.

For $N\in\nats$, let
$\Sigma(N) = \{\tilde{i}\in {\nats}_0^d : i_1 + \cdots + i_d = N\}$,
where $\tilde{i} = (i_1,\ldots,i_d)$. 
Suppose $p\in R$ is a homogeneous polynomial of degree $N\in\nats$
such that $p\neq 0 \pmod{(x_1^k,\ldots,x_d^k)}$, so
$p = \sum_{\tilde{i}\in s_p(N)}a_{\tilde{i}}{\tilde{x}}^{\tilde{i}}$ for
some nonempty subset $s_p(N) \subseteq \Sigma(N)$ where $a_{\tilde{i}}\neq 0$
for each $\tilde{i}\in s_p(N)$. If $\tilde{1} = (1,1,\ldots,1)\in {\nats}^d$
then we can, needless to say, assume $\tilde{i}\leq (k-1)\tilde{1}$ in the
componentwise partial order on ${\ints}^d$.
\begin{claim}
\label{clm:ij}
Let $\tilde{i},\tilde{j}\in ({\nats}_0)^d$.
If $\sum_{\ell=1}^d(i_{\ell}+j_{\ell}) = d(k-1)$, then either
$\max_{\ell}(i_{\ell}+j_{\ell}) \geq k$ or
$\tilde{i}+\tilde{j} = (k-1)\tilde{1}$.
\end{claim}
This follows from the pigeonhole principle.  For any $\tilde{i}\in s_p(N)$ we have by Claim~\ref{clm:ij} that 
\begin{equation}
\label{eqn:mod-max}  
  {\tilde{x}}^{(k-1)\tilde{1} - \tilde{i}}p 
  \equiv a_{\tilde{i}}{\tilde{x}}^{(k-1)\tilde{1}}\pmod{(x_1^k,\ldots,x_d^k)}.
\end{equation}
In particular, we have ${\tilde{x}}^{(k-1)\tilde{1} - \tilde{i}} \not\in I$
for any $\tilde{i}\in s_p(N)$. Multiplying 
(\ref{eqn:mod-max}) through by $x_i$, we further see that
$x_i{\tilde{x}}^{(k-1)\tilde{1} - \tilde{i}}p \in (x_1^k,\ldots,x_d^k)$, and
so $x_i{\tilde{x}}^{(k-1)\tilde{1} - \tilde{i}} \in I$ for each
$i\in\{1,\ldots,d\}$. Thus, we have
\begin{claim}
\label{clm:1-elt-socle}  
For each $\tilde{i}\in s_p(N)$ we have
${\tilde{x}}^{(k-1)\tilde{1} - \tilde{i}}\in(I:\m)\setminus I = \soc(I)\setminus I$.
\end{claim}
In what follows, $p$ is our given homogeneous polynomial of degree $N$.
We now briefly describe the explicit form of the unique element in $\soc(I)\setminus I$
stated in Observation~\ref{obs:soc-monomial} below and thereby re-establish the 
Gorensteinness of $I$.

Suppose $f\in\soc(I)\setminus I$, so $f\not\in I$ and $x_if\in I$ for each $x_i$,
that is $fp\not\in (x_1^k,\ldots,x_d^k)$ and $x_ifp\in (x_1^k,\ldots,x_d^k)$.
This implies that
$fp \equiv {\tilde{x}}^{(k-1)\tilde{1}}\pmod{(x_1^k,\ldots,x_d^k)}$, where
$f$ has been scaled appropriately.
Clearly we can write $f = \sum_{i=0}^{d(k-1)-N}f_i$
where each $f_i$ is homogeneous of degree $i$,
we have $f = f_{d(k-1)-N} + f'$ and so $f'p\in (x_1^k,\ldots,x_d^k)$ or
$f'\in I$, and
$f_{d(k-1)-N}p \equiv {\tilde{x}}^{(k-1)\tilde{1}}\pmod{(x_1^k,\ldots,x_d^k)}$.
We can therefore drop the subscript and further assume $f$ is homogeneous of
degree $d(k-1)-N$ and so $fp$ is homogeneous of degree $d(k-1)$.
By Claim~\ref{clm:ij} we can further assume $f$ to be of the form
$f = \sum_{\tilde{i}\in s_p(N)}b_{\tilde{i}}{\tilde{x}}^{(k-1)\tilde{1} - \tilde{i}}$
modulo the ideal $I$. Again by Claim~\ref{clm:ij} we get
$fp \equiv
\sum_{\tilde{i}\in s_p(N)}a_{\tilde{i}}b_{\tilde{i}}{\tilde{x}}^{(k-1)\tilde{1}}
\pmod{(x_1^k,\ldots,x_d^k)}$.
We then have that $f$ must have the mentioned form
$f = \sum_{\tilde{i}\in s_p(N)}b_{\tilde{i}}{\tilde{x}}^{(k-1)\tilde{1} - \tilde{i}}$
where $\sum_{\tilde{i}\in s_p(N)}a_{\tilde{i}}b_{\tilde{i}} = 1$.
By (\ref{eqn:mod-max}), we also have that
for distinct $\tilde{i},\tilde{j} \in s_p(N)$,

\begin{equation}
    \label{eqn:socle-monomial}
(a_{\tilde{i}}{\tilde{x}}^{(k-1)\tilde{1} - \tilde{j}}
- a_{\tilde{j}}{\tilde{x}}^{(k-1)\tilde{1} - \tilde{i}})p\in (x_1^k,\ldots,x_d^k),
\end{equation}

and so $a_{\tilde{i}}{\tilde{x}}^{(k-1)\tilde{1} - \tilde{j}}
- a_{\tilde{j}}{\tilde{x}}^{(k-1)\tilde{1} - \tilde{i}} \in I$.
By this, together with Claim~\ref{clm:1-elt-socle}, we then have 
$(I:\m) = \soc(I) = \{{\tilde{x}}^{(k-1)\tilde{1} - \tilde{i}}\}$
modulo $I$ for each $\tilde{i}\in s_p(N)$.
This re-establishes the known
fact that $I$ is Gorenstein, but more importantly for our discussion it
yields an explicit $K$-algebra property of $R/I$.
\begin{observation}
\label{obs:soc-monomial}
For a homogeneous, zero-dimensional Gorenstein ideal 
$I = ((x_1^k,\ldots,x_d^k):p)$, where $p$ and $s_p(N)$ are as before,  
the socle of $I$ is always generated by a single monomial of $R$. 
Indeed, $\soc(I) = \{{\tilde{x}}^{(k-1)\tilde{1} - \tilde{i}}\}$ 
modulo $I$ for each $\tilde{i}\in s_p(N)$.
\end{observation}
Recall that a {\em term order} on the monoid $[x_1,\ldots,x_d]$ is a linear,
or total, order $\preceq$ of the monomials that respects the multiplication,
so $u\preceq v \Rightarrow uw\preceq vw$ for any $u,v,w\in [x_1,\ldots,x_d]$.
We will here use the {\em lexicographical term order (LEX)} on the monomials in
$[x_1,\ldots,x_d]$ that is determined by $x_1 \prec x_2 \prec \cdots \prec x_d$.
So one first compares the exponents of $x_d$ in the monomials, and then, in case of 
equality, one compares the exponents of $x_{d-1}$ and so on\footnote{the
  {\em degree lexicographical term order (DEGLEX)} would work exactly
  the same way for us since $p$ is homogeneous.}.

\begin{convention}
The natural partial order on $[x_1,\ldots,x_d]$ will
be denoted by ``$\leq$'', as we have done so far, whereas ``$\preceq$'' will
mean a total term ordering on $[x_1,\ldots,x_d]$ (for the most part the
aforementioned LEX ordering).
\end{convention}

Since $I$ is generated by homogeneous polynomials from $R$, every
Gr\"{o}bner basis for $I$ in relation to any term order, in particular LEX,
consists of homogeneous polynomials. This can be seen by explicitly
carrying though Buchberger's Algorithm, which generalizes both the Euclidean
Algorithm and Gaussian Elimination in linear algebra into the context of
multivariate polynomials~\cite[Chapter 15]{Eisenbud}. Since $R/I$ is a zero
dimensional ring, and therefore finite dimensional as a vector space over $K$. Then each
Gr\"{o}bner basis for $I$ contains $d$ polynomials, the leading terms
of which are pure powers of $x_1,\ldots,x_d$ with each power at most $k$.

Let $\tilde{\mu}\in s_p(N)$ be such that ${\tilde{x}}^{\tilde{\mu}}$ is
the largest monomial appearing as a summand in our given polynomial 
$p$ with respect to the LEX term ordering. Since $p$ is homogeneous,
${\tilde{x}}^{(k-1)\tilde{1} - \tilde{\mu}}$
is the least monomial representing the generator of $\soc(I)/I$ as stated 
in Observation~\ref{obs:soc-monomial}. 
\begin{proposition}
\label{prp:least-irreducible}
With the above conventions and notations,
${\tilde{x}}^{(k-1)\tilde{1} - \tilde{\mu}}$ is an irreducible monomial 
with respect to LEX; that is, it cannot be reduced further using the 
Gr\"{o}bner basis obtained from LEX.
\end{proposition}
\begin{proof}
Suppose ${\tilde{x}}^{(k-1)\tilde{1} - \tilde{\mu}} - f \in I$ for some
polynomial $f$, where each monomial term in $f$ is strictly less than 
${\tilde{x}}^{(k-1)\tilde{1} - \tilde{\mu}}$ with respect to LEX. Since our 
Gr\"{o}bner basis with respect to LEX consists of homogeneous polynomials, then
$f$ is homogeneous of the same degree $\Delta:= d(k-1) - N$ as 
${\tilde{x}}^{(k-1)\tilde{1} - \tilde{\mu}}$. Further, we can assume
each monomial term appearing in $f$ not to be in $(x_1^k,\ldots,x_d^k)$ 
and hence of the form ${\tilde{x}}^{(k-1)\tilde{1} - \tilde{j}}$ for 
some $\tilde{j}\in\Sigma(N)$. By definition of $\tilde{\mu}$ we have
$\tilde{x}^{\tilde{i}}\preceq \tilde{x}^{\tilde{\mu}}$ for each 
$\tilde{i}\in s_p(N)$ and by our assumption on $f$ we have
${\tilde{x}}^{(k-1)\tilde{1} - \tilde{j}}\prec {\tilde{x}}^{(k-1)\tilde{1} - \tilde{\mu}}$
for every monomial ${\tilde{x}}^{(k-1)\tilde{1} - \tilde{j}}$ in $f$,
and therefore $\tilde{x}^{\tilde{i}}\preceq \tilde{x}^{\tilde{\mu}}\prec {\tilde{x}}^{\tilde{j}}$.
In particular, the monomials ${\tilde{x}}^{\tilde{i}}$ and ${\tilde{x}}^{\tilde{j}}$ are
distinct. By Claim~\ref{clm:ij} we then have $pf\in I$ and hence 
${\tilde{x}}^{(k-1)\tilde{1}} =  p{\tilde{x}}^{(k-1)\tilde{1} - \tilde{\mu}} \in I$,
a blatant contradiction.
\end{proof}
By the above Proposition~\ref{prp:least-irreducible}, every monomial in 
\begin{equation}
\label{eqn:B(I)}
B(I) :=
\{ {\tilde{x}}^{\tilde{i}} : \tilde{0}\leq\tilde{i}\leq (k-1)\tilde{1} - \tilde{\mu}\}
\end{equation}
is fully reduced with respect to $GB(I)$, and hence $B(I)$ 
forms part of a basis for $R$ modulo $I$ as a vector space over $K$. 

We now discuss a way of presenting our ideal $I = ((x_1^k,\ldots,x_d^k):p)$
as an inverse ideal of particular differential operators, provided that certain criteria 
are met. Since $p$ is 
homogeneous of degree $N$, then $R/I = \bigoplus_{i=0}^{\Delta}(R/I)_i$ as a graded 
$K$-algebra, where $\Delta = d(k-1) - N$. Consider now a nonzero homogeneous 
element $f \in (R/I)_i$ of degree $i\in \{0,\ldots, \Delta\}$. By definition of $I$
as an ideal quotient, we have $pf\equiv g \pmod{(x_1^k,\ldots,x_d^k)}$ where
$g\in \spn_K(\{{\tilde{x}}^{\tilde{i}} :  \tilde{0}\leq\tilde{i}\leq (k-1)\tilde{1}\})$ 
is a nonzero homogeneous polynomial of degree $N+i\leq d(k-1)$. 
Assuming $g$ is scaled appropriately, there exists a monomial ${\tilde{x}}^{\tilde{r}}$ of 
degree $d(k-1)-N-i$ (not necessarily unique!) such that 
${\tilde{x}}^{\tilde{r}}g \equiv {\tilde{x}}^{(k-1)\tilde{1}} \pmod{(x_1^k,\ldots,x_d^k)}$,
and hence 
$({\tilde{x}}^{\tilde{r}}f)p \equiv {\tilde{x}}^{(k-1)\tilde{1}} \pmod{(x_1^k,\ldots,x_d^k)}$.
Therefore ${\tilde{x}}^{\tilde{r}}f\not\in I$, yet it is homogeneous of degree 
$\Delta = d(k-1) - N$, the maximum degree of an element in $R/I$ and hence is contained
in the one-dimensional $K$-vector space $(R/I)_{\Delta}$ generated by the sole element (up to $K$-scaling) of
the socle of $I$ that is not in $I$ (as in Observation~\ref{obs:soc-monomial}).
As neither ${\tilde{x}}^{\tilde{r}}$ nor $f$ is contained in $I$, we have that 
the $K$-linear function 
$(R/I)_i \rightarrow (R/I)_{\Delta} = \soc(I)/I \cong K$,
$\overline{u} \mapsto \overline{fu}$ is nonzero.
Since $f$ was an arbitrary nonzero element of $(R/I)_i$, we see
that the bilinear form
$(R/I)_i \times (R/I)_{{\Delta}-i} \rightarrow (R/I)_{\Delta} = \soc(I)/I \cong K$,
$(u,v)\mapsto uv$  is nondegenerate. Using the same terminology as
in Observation~\ref{obs:soc-monomial}, we summarize in the following observation.
We believe that this observation is well-known folklore, but since we have not 
seen it stated in the following specific form, we included it so it is directly
compatible with Lemma 13.4.7 in~\cite{CoxLittleSchenck}. 
\begin{observation}
\label{obs:bi-linear}
For our homogeneous, zero-dimensional Gorenstein ideal 
$I = ((x_1^k,\ldots,x_d^k):p)$ of $R = K[x_1,\ldots,x_d]$,
the graded $K$-algebra $R/I = \bigoplus_{i=0}^{\Delta} (R/I)_i$
satisfies the following conditions:
(i) $(R/I)_0 \cong (R/I)_{\Delta} \cong K$,
(ii) $R/I$ is generated by $\overline{x}_1,\ldots,\overline{x}_d \in (R/I)_1$
as a $K$-algebra, and
(iii) the bilinear form
$(R/I)_i \times (R/I)_{{\Delta}-i} \rightarrow (R/I)_{\Delta} = \soc(I)/I \cong K$,
$(u,v)\mapsto uv$  is nondegenerate.
\end{observation}
By Observation~\ref{obs:bi-linear} and \cite[Lemma~13.4.7, p.~659]{CoxLittleSchenck}, we then have
for the ring $R/I$ the following corollary.
\begin{corollary}
\label{cor:IS-CLS}  
Let $I$ and $R$ be as in Observation~\ref{obs:bi-linear}.
Let $t_1,\ldots,t_d$ be indeterminates and
\[
P(t_1,\ldots,t_d) = (t_1\overline{x}_1 + \cdots + t_d\overline{x}_d)^{\Delta}
\in (R/I)_{\Delta}[t_1,\ldots,t_d] \cong K[t_1,\ldots,t_d]
\]
which is a polynomial over $K$ via the $K$-vector space isomorphism
${\tilde{x}}^{(k-1)\tilde{1} - \tilde{\mu}}\mapsto 1\in K$. In this case the ideal
$I$ has the form
\[
I = \left\{f\in R = K[x_1,\ldots,x_d] :
f\left(\frac{\partial}{\partial t_1},
\cdots,\frac{\partial}{\partial t_d}\right)P(t_1,\ldots,t_d) = 0\right\}.
\]
\end{corollary}
Note that in the above corollary, our ideal $I$ is equivalent to the ideal 
$(0:_S M)$ given in~\cite[Thm 21.6]{Eisenbud}, where $M$ is the $S$-module
generated by the polynomial $P$ (with the IS defined as in~\cite[p.~526]{Eisenbud}).
In Corollary~\ref{cor:IS-CLS} however, the form of the polynomial $P = P(t_1,\ldots,t_d)$ 
is provided. This is exactly what will be useful for us.

As stated in the above Corollary~\ref{cor:IS-CLS}, then {\em any}
homogeneous, zero-dimensional Gorenstein ideal of $R = K[x_1,\ldots,x_d]$, 
necessarily of the form $I = ((x_1^k,\ldots,x_d^k):p)$, 
can indeed in this way be viewed as the inverse ideal of
the singleton set $M' = \{P\}$ containing the polynomial 
$P = P[t_1,\ldots,t_d]\in K[t_1,\ldots,t_d]$
as stated in Observation~\ref{obs:I=J}.

\begin{example}
\label{exa:motivating-2}
Continuing with Example~\ref{exa:motivating}, consider the ideal $I = (x^3, y^2 - xy)$
of $R = K[x,y]$ that can be written as $I = ((x^4,y^4) : p)$ where 
$p = xy^2 + x^2y + x^3$ is of homogeneous degree $3$ and is
therefore homogeneous, zero-dimensional and Gorenstein.
Here, the generators 
$y^2 - xy$ and $x^3$ form a Gr\"{o}bner basis for $I$ with respect to LEX,
and the unique monomial generator for the socle of $I$ modulo $I$ is
given by ${\tilde{x}}^{(k-1)\tilde{1} - \tilde{\mu}} = x^2y$ and so $\Delta = 3$. 
In this case $B(I)$ from (\ref{eqn:B(I)}) forms the entire basis for $R/I$ 
as a $K$-vector space. As the quotient ideal $I = ((x^4,y^4) : p)$ of $R$, 
it satisfies the three conditions in Observation~\ref{obs:bi-linear}. 
By reducing in terms of the Gr\"{o}bner basis $GB(I) = \{y^2 - xy,x^3\}$, we obtain
$(t_1\overline{x} + t_2\overline{y})^3
= (3t_1^2t_2 + 3t_1t_2^2 + t_2^3)\overline{x}^2\overline{y}$ and 
hence 
\[
P = P(t_1,t_2) = 3t_1^2t_2 + 3t_1t_2^2 + t_2^3 \in K[t_1,t_2]
\]
in this case. It is here easy to check directly that those $f\in K[x,y]$ such that 
$f(\frac{\partial}{\partial t_1},\frac{\partial}{\partial t_2})P = 0$
form the ideal $I = (x^3,y^2-xy)$, consistent with
Corollary~\ref{cor:IS-CLS}. Note that $P$ is not a monomial in $t_1,\ldots,t_d$,
although it is a homogeneous polynomial of degree $3$ containing the
same number of terms as $p$ does, and the first monomial in $P$
corresponds to the monomial that generates the socle $\soc(I)/I$ of $R/I$. 
\end{example}

\begin{example}
Consider now a corresponding zero-dimensional Gorenstein 
monomial ideal $I = (x^3, y^2)$ of $R = K[x,y]$. Here $I$ can be written
in the form of $I = ((x^3,y^3) : y)$, and so also in this case the
ideal $I$ satisfies the three conditions in  
Observation~\ref{obs:bi-linear} (which can also be seen directly in this case). 
Also in this case, $B(I)$ from (\ref{eqn:B(I)}) forms the entire 
basis for $R/I$ as a $K$-vector space. 
Here we readily obtain the unique monomial generator for 
$\soc(I)$ modulo $I$ as ${\tilde{x}}^{(k-1)\tilde{1} - \tilde{\mu}} = x^2y$,
and so $\Delta = 3$. Also we get 
$(t_1\overline{x} + t_2\overline{y})^3 = 3t_1^2t_2\overline{x}^2\overline{y}$,
and hence
\[
P = P(t_1,t_2) = 3t_1^2t_2\in K[t_1,t_2]
\]
in this case, which is a scalar multiple of the very monomial that
corresponds to the unique monomial in $\soc(I)\setminus I$. Hence,
the ideal $I$ in this example falls squarely in what is stated in 
Observation~\ref{obs:I=J}, since $I = (x^3, y^2)$ is here
the inverse ideal $I = {M'}^{-1} = \ann(M')$ of $M' = \{t_1^2t_2\}$,
a monomial set corresponding to the singleton set $M = \{x^2y\}$
and the docle $\doc(I) = M$.
As in previous example $x = \frac{\partial}{\partial t_1}$
and $y = \frac{\partial}{\partial t_2}$ acts as differential operators.
\end{example}

\begin{example}
Consider the ideal $I = ((x^{10},y^{10}):p)$ of $R = K[x,y]$
where $p = y^6 + x^3y^3 + x^5y$. In this case we have 
$I = (xy^6 - 2x^2y^5 - 2x^3y^4 - x^4y^3 + 2x^5y^2 + x^6y + 3x^7, y^7 - x^2y^5 + x^7)$
and $I$ has a Gr\"{o}bner basis with respect to LEX given as follows 
(courtesy of the software package Macaulay2~\cite{Macaulay2}):
\begin{eqnarray*}
GB(I) & = & \{ y^7 - x^2y^5 + x^5, \\
      & = &  xy^6 - 2x^2y^5 - 2x^3y^4 - x^4y^3 + 2x^5y^2 + x^6y + 3x^7, \\
      & = &  x^3y^5 +x^4y^4 - x^6y^2 - x^7y - x^8, \\
      & = &  x^5y^4 - x^8 y, \\
      & = &  x^{10}\} .
\end{eqnarray*}
From $GB(I)$, we see that $R/I$ is of dimension $49$ over $K$. 
Here ${\tilde{x}}^{(k-1)\tilde{1} - \tilde{\mu}} = x^9y^3$, 
and $B(I)$ from (\ref{eqn:B(I)}) contains $40$ elements, and
so it forms a proper subset of a $K$-basis for $R/I$ formed by
the irreducible monomials with respect to $GB(I)$. Here the unique
monomial generator for $\soc(I)$ modulo $I$ is 
${\tilde{x}}^{(k-1)\tilde{1} - \tilde{\mu}} = x^9y^3$, and so $\Delta = 12$.
In this case we obtain, using the reductions from the above Gr\"{o}bner basis
$GB(I)$ with respect to LEX implemented in the software package Macaulay2~\cite{Macaulay2}, 
that 
\[
(t_1\overline{x} + t_2\overline{y})^{12} 
= (220t_1^9t_2^3 + 924t_1^6t_2^6 + 495t_1^4t_2^8)\overline{x}^9\overline{y}^3,
\]
and so 
\[
P = P(t_1,t_2) = 220t_1^9t_2^3 + 924t_1^6t_2^6 + 495t_1^4t_2^8 \in K[t_1,t_2].
\]
Note that $P$ and $p$ contain the same numbers of terms. 
\end{example}

We gather some of our notes from the previous three examples as follows:
(i) In the first two examples $B(I)$ is the entire basis for $R/I$, 
but not in the third example. 
(ii) Only in the second example is $P = P(t_1,t_2)$ a monomial corresponding 
to the unique monomial generator of $\soc(I)/I$, which in that case is the docle $\doc(I)$
since $I$ is itself a monomial ideal.
(iii) In all examples $P(t_1,t_2)$ is a homogeneous polynomial of degree 
$\Delta = d(k-1) - N$ and is further a $K$-linear combination of 
the terms ${\tilde{x}}^{(k-1)\tilde{1} - \tilde{i}}$
where $\tilde{i}\in s_p(N)$ for each of the given cases of the homogeneous polynomial $p$. This
is not a coincidence. We conclude this section with a proposition (along with a corollary) where this  
observation is made precise. 
\begin{proposition}
\label{prp:P}
Let $I$, $p$, $R$ and $P(t_1,\ldots,t_d) \in K[t_1,\ldots,t_d]$ be 
as in Corollary~\ref{cor:IS-CLS}, so we have
$p = \sum_{\tilde{i}\in s_p(N)}a_{\tilde{i}}{\tilde{x}}^{\tilde{i}}$
where $a_{\tilde{i}} \neq 0$ for each $\tilde{i}\in s_p(N)$.
In this case, $P = P(t_1,\ldots,t_d)$ from Corollary~\ref{cor:IS-CLS} is
(up to a non-zero $K$-scalar) given by
\[
P = \sum_{\tilde{i}\in s_p(N)}
\binom{\Delta}{k-1-i_1 \:\cdots\: k-1-i_d}
a_{\tilde{i}}{\tilde{t}}^{(k-1)\tilde{1} - \tilde{i}} 
\in K[t_1,\ldots,t_d].
\]
\end{proposition}
\begin{proof}
Let ${\tilde{x}}^{\tilde{j}}$ be a monomial of $R = K[x_1,\ldots,x_d]$ of 
degree $\Delta = d(k-1) - N$, so $\tilde{j}\in \Sigma(\Delta)$. Clearly 
${\tilde{x}}^{\tilde{j}}p \equiv 0 \pmod{(x_1^k,\ldots x_d^k)}$
if at least one $x_{\ell}\geq k$, or if each $x_{\ell}\leq k-1$ and 
$\tilde{j} + \tilde{i} \neq (k-1)\tilde{1}$ for each $\tilde{i}\in s_p(N)$
by~Claim~\ref{clm:ij}. Otherwise,
$\tilde{j} + \tilde{i} = (k-1)\tilde{1}$ for exactly one $\tilde{i}\in s_p(N)$
in which case we have 
\[
{\tilde{x}}^{\tilde{j}}p \equiv 
a_{(k-1)\tilde{1} - \tilde{j}}{\tilde{x}}^{(k-1)\tilde{1}}
\pmod{(x_1^k,\ldots x_d^k)}.
\]
Using the multinomial expansion on 
$(t_1\overline{x}_1 + \cdots + t_d\overline{x}_d)^{\Delta}$ first
and then secondly using (\ref{eqn:mod-max}) we then obtain 
\begin{eqnarray*}
\lefteqn{(t_1\overline{x}_1 + \cdots + t_d\overline{x}_d)^{\Delta}p } \\
& & 
\equiv\left(\sum_{\tilde{i}\in s_p(N)}
\binom{\Delta}{k-1-i_1 \:\cdots\: k-1-i_d}a_{\tilde{i}}{\tilde{t}}^{(k-1)\tilde{1} - \tilde{i}}\right)
      {\tilde{x}}^{(k-1)\tilde{1}} \pmod{(x_1^k,\ldots x_d^k)} \\
& &       
\equiv\left(\sum_{\tilde{i}\in s_p(N)}
\binom{\Delta}{k-1-i_1 \:\cdots\: k-1-i_d}{(\tilde{t}\tilde{x}})^{(k-1)\tilde{1} - \tilde{i}}\right)p \pmod{(x_1^k,\ldots x_d^k)}
\end{eqnarray*}
and so by definition of $I$ that
\[
(t_1\overline{x}_1 + \cdots + t_d\overline{x}_d)^{\Delta} \equiv 
\sum_{\tilde{i}\in s_p(N)}
\binom{\Delta}{k-1-i_1 \:\cdots\: k-1-i_d}({\tilde{t}\tilde{x}})^{(k-1)\tilde{1} - \tilde{i}}
\pmod{I}.
\]
By~(\ref{eqn:socle-monomial}) we then obtain, for any fixed $\tilde{\ell}\in s_p(N)$, that
$a_{\tilde{\ell}}{\tilde{x}}^{(k-1)\tilde{1} - \tilde{i}} \equiv
a_{\tilde{i}}{\tilde{x}}^{(k-1)\tilde{1} - \tilde{\ell}} \pmod{I}$ and hence,
by multiplying through the above equation by $a_{\tilde{\ell}}$, we obtain 
\[
a_{\tilde{\ell}}(t_1\overline{x}_1 + \cdots + t_d\overline{x}_d)^{\Delta} \equiv 
\left(\sum_{\tilde{i}\in s_p(N)}
\binom{\Delta}{k-1-i_1 \:\cdots\: k-1-i_d}a_{\tilde{i}}
{\tilde{t}}^{(k-1)\tilde{1} - \tilde{i}}\right)
{\tilde{x}}^{(k-1)\tilde{1} - \tilde{\ell}}\pmod{I}
\]
which completes our proof.
\end{proof}

For any polynomial $Q = Q(t_1,\ldots,t_d)\in K[t_1,\ldots,t_d]$ one can form
the ideal of $R = K[x_1,\ldots,x_d]$ as 
\begin{equation}
\label{eqn:ann-partial}
\ann_{\partial}(Q) := \left\{f\in K[x_1,\ldots,x_d] :
f\left(\frac{\partial}{\partial t_1},
\cdots,\frac{\partial}{\partial t_d}\right)Q(t_1,\ldots,t_d) = 0\right\}
\end{equation}
as the inverse ideal of the singleton set $M' = \{Q\}$, as for the specific 
polynomial $P$ in Corollary~\ref{cor:IS-CLS} (and equivalent to
$(0:_S M)$ where $M$ is the $S$-module generated by $Q$ from~\cite[Thm 21.6]{Eisenbud}.) 
Given $k$ and the Krull dimension $d$ of $R$, then for any homogeneous polynomial 
$p = \sum_{\tilde{i}\in s_p(N)}a_{\tilde{i}}{\tilde{x}}^{\tilde{i}}$
of degree $N$,
where $a_{\tilde{i}} \neq 0$ for each $\tilde{i}\in s_p(N)$, we can 
always define the {\em $d;k$-antipodal polynomial} 
$p_{d;k}^{{{\blacktriangle}}}\in R$ from Proposition~\ref{prp:P} in 
terms of $x_1,\ldots,x_d$ by
\begin{equation}
\label{eqn:antipodal}
p_{d;k}^{{{\blacktriangle}}} = 
\sum_{\tilde{i}\in s_p(N)}
\binom{\Delta}{k-1-i_1 \:\cdots\: k-1-i_d}
a_{\tilde{i}}{\tilde{x}}^{(k-1)\tilde{1} - \tilde{i}}.
\end{equation}
We note that $p_{d;k}^{{{\blacktriangle}}}$ is homogeneous of degree $\Delta = d(k-1) - N$.
By Proposition~\ref{prp:P} we have then for any 
homogeneous, zero-dimensional, Gorenstein ideal 
$((x_1^k,\ldots,x_d^k):p)\subseteq R$ that
\begin{equation}
\label{eqn:Gorenstein-ann}
((x_1^k,\ldots,x_d^k):p) = \ann_{\partial}(p_{d;k}^{{{\blacktriangle}}}).
\end{equation}
Note that if $I$ is a zero-dimensional Gorenstein monomial ideal,
then we already had by Observation~\ref{obs:I=J} that 
$I = \ann_{\partial}(q)$, where $q$ is the unique monomial generator  
of $\soc(I)$ modulo the ideal $I$. What we have now by (\ref{eqn:Gorenstein-ann})
is the stronger if-and-only-if statement when $K$ is of characteristic zero.
Recall that if $I\subseteq R$ is a Gorenstein ideal, then $\soc(I)/I$ is always
a one dimensional vector space over $K$, and so $\soc(I) \pmod I$ contains a unique 
element up to scaling by $K^\times$.
\begin{corollary}
\label{cor:ann-iff-monomial}
If our ground field $K$ has characteristic zero, then
for any homogeneous, zero-dimensional, Gorenstein ideal $I$ of 
$R = K[x_1,\ldots,x_d]$ we have that 
$I = \ann_{\partial}(q)$, where $q$ is the unique generator 
of $\soc(I) \pmod{I}$ up to $K$-scaling, if and only if $I$ is a monomial ideal. 
\end{corollary}

\section{A more universal presentation}
\label{sec:universal}

In this short section we present a few observation about the form of the
polynomial $P$ in Corollary~\ref{cor:IS-CLS}. Namely, the same
conclusion can be achieved with a more universal form for the 
polynomial $P$.

As before, let $R = K[x_1,\ldots,x_d]$ be the polynomial ring in $d$ variables,
and consider a homogeneous zero dimensional ideal $I$ of $R$ (not
necessarily Gorenstein as in the previous section).
As in Corollary~\ref{cor:IS-CLS}, let $t_1,\ldots,t_d$ be indeterminates.
Let $\overline{x}_i$ denote the coset of $x_i$ in $R/I$ for each $i$ and
$g\in R$ be a homogeneous polynomial of degree $N$. For any polynomial
$f\in K[z]$
we readily obtain that
\begin{equation}
  \label{eqn:gen-diff}
g\left(\frac{\partial}{\partial t_1},
\cdots,\frac{\partial}{\partial t_d}\right)f(t_1\overline{x}_1 + \cdots + t_d\overline{x}_d)
=
f^{(N)}(t_1\overline{x}_1 + \cdots + t_d\overline{x}_d)
g(\overline{x}_1,\ldots,\overline{x}_d),
\end{equation}
where $f^{(N)}$ denotes the $N$-th derivative of $f\in K[z]$ as a polynomial
in one variable $z$. For any such polynomial $f\in K[z]$ we therefore have
that
\begin{equation}
  \label{eqn:diff-includ}
I\subseteq \left\{g\in K[x_1,\ldots,x_d] :
g\left(\frac{\partial}{\partial t_1},
\cdots,\frac{\partial}{\partial t_d}\right)f(t_1\overline{x}_1 + \cdots + t_d\overline{x}_d) = 0\right\}.
\end{equation}
Since our ideal $I$ of $R$ is homogeneous and zero-dimensional, 
the ring $R/I$ is finite dimensional as a vector space over $K$ and 
$R/I = \bigoplus_{i=0}^{\Delta}(R/I)_i$ is graded. Thus,
$(t_1\overline{x}_1 + \cdots + t_d\overline{x}_d)^n = 0$ in 
$\overline{R} = K[t_1,\ldots,t_d]\otimes_K R/I$ for any $n > {\Delta}$.
Consequently, if $f(z) = \sum_{n\geq 0}a_nz^n \in K[[z]]$
is a formal power series over $K$, then 
$f(t_1\overline{x}_1 + \cdots + t_d\overline{x}_d) = \sum_{n = 0}^{\Delta} 
a_n(t_1\overline{x}_1 + \cdots + t_d\overline{x}_d)^n \in \overline{R}$.

Now we consider the case $n\leq {\Delta}$. As
$R/I = \bigoplus_{i=0}^{\Delta}(R/I)_i$ is graded
with maximum degree of $\Delta$ (and contains a nonzero element of degree $\Delta$), 
 not all monomials in in $\overline{x}_1,\ldots,\overline{x}_d$ can be zero,
and hence $(t_1\overline{x}_1 + \cdots + t_d\overline{x}_d)^n \neq 0$ in 
$\overline{R}$. Together with (\ref{eqn:diff-includ}), we gather these 
intermediate observations in the following.
\begin{observation}
\label{obs:zero-not}
If $I$ is a homogeneous zero dimensional ideal of $R = K[x_1,\ldots,x_d]$, 
and therefore $R/I = \bigoplus_{i=0}^{\Delta}(R/I)_i$ is a graded
$K$-algebra with 
$\Delta\in\nats$ as the maximum degree among its homogeneous elements,  
$(t_1\overline{x}_1 + \cdots + t_d\overline{x}_d)^n \neq 0$
in $\overline{R} = K[t_1,\ldots,t_d]\otimes_K R/I$ if and only if $n\leq {\Delta}$.

In particular, for any formal power series
$f(z) = \sum_{n\geq 0}a_nz^n \in K[[z]]$,
we therefore have that
$f(t_1\overline{x}_1 + \cdots + t_d\overline{x}_d) \in \overline{R}$,
and further $I\subseteq \ann_{\partial}(f)$ for any such formal power series,
where the differential operator is applied in the obvious way.
\end{observation}

In the case $n\leq {\Delta}$, we note further that if $n < N$ then, 
since $g$ is homogeneous of degree $N$, we have by (\ref{eqn:gen-diff}) that
\[
g\left(\frac{\partial}{\partial t_1},
\cdots,\frac{\partial}{\partial t_d}\right)(t_1\overline{x}_1 + \cdots + t_d\overline{x}_d)^n
= 0,
\]
and if $N\leq n\leq \Delta$, then likewise we get by (\ref{eqn:gen-diff}) that
\begin{eqnarray*}
\lefteqn{g\left(\frac{\partial}{\partial t_1},
\cdots,\frac{\partial}{\partial t_d}\right)(t_1\overline{x}_1 + \cdots + t_d\overline{x}_d)^n } \\
& & = \frac{n!}{(n-N)!}(t_1\overline{x}_1 + \cdots + t_d\overline{x}_d)^{n-N}
g(\overline{x}_1,\ldots,\overline{x}_d) \in K[t_1,\ldots,t_d]\otimes_K R_n\subseteq \overline{R},
\end{eqnarray*}
a polynomial in $\overline{x}_1,\ldots,\overline{x}_d$ of degree $n$ with coefficients 
in $K[t_1,\ldots,t_d]$.

Therefore, for any formal power series $f(z) = \sum_{n\geq 0}a_nz^n \in K[[z]]$,
we have by (\ref{eqn:gen-diff}) and Observation~\ref{obs:zero-not} that
\begin{eqnarray*}
\lefteqn{g\left(\frac{\partial}{\partial t_1},
\cdots,\frac{\partial}{\partial t_d}\right)f(t_1\overline{x}_1 + \cdots + t_d\overline{x}_d) } \\
 & & = f^{(N)}(t_1\overline{x}_1 + \cdots + t_d\overline{x}_d)
g(\overline{x}_1,\ldots,\overline{x}_d) \\ 
 & & = \left(\sum_{n\geq N}a_n\frac{n!}{(n-N)!}(t_1\overline{x}_1 + \cdots + t_d\overline{x}_d)^{n-N}\right)
 g(\overline{x}_1,\ldots,\overline{x}_d) \\
 & & = \sum_{n\geq N}^{\Delta}\left(a_n\frac{n!}{(n-N)!}(t_1\overline{x}_1 + \cdots + t_d\overline{x}_d)^{n-N}
 g(\overline{x}_1,\ldots,\overline{x}_d)\right),
 \end{eqnarray*}
which equals zero in $\overline{R}$ if and only if each summand equals zero. 
Suppose now that $I$ is Gorenstein. In this case, the last summand of degree $\Delta$ in 
$\overline{x}_1,\ldots,\overline{x}_d$ must be zero as well. Since $I$ is Gorenstein
the $\Delta^{\text{th}}$ summand $K[t_1,\ldots,t_d]\otimes_K(R/I)_{\Delta}$ is a cyclic 
$K[t_1,\ldots,t_d]$-module generated by the unique monomial in the socle of $R/I$,
as in Corollary~\ref{cor:IS-CLS}. As this last summand is zero if and only if
its coefficient in $K[t_1,\ldots,t_d]$ is zero, we have the latter
equality of Corollary~\ref{cor:IS-CLS}. In summary, the polynomial $P = P(t_1,\ldots,t_d)$
in Corollary~\ref{cor:IS-CLS} can be replaced by $f(t_1\overline{x}_1 + \cdots + t_d\overline{x}_d)$,
where $f(z) = \sum_{n\geq 0}a_nz^n \in K[[z]]$ is any formal power series with non-zero coefficients. 
We gather this in the following final Proposition.
\begin{proposition}
\label{prp:any-f}
Let $I$ be a homogeneous zero-dimensional Gorenstein ideal
of $R$. Let $f(z) = \sum_{n\geq 0}a_nz^n \in K[[z]]$ be a
a formal power series with non-zero coefficients and 
$t_1,\ldots,t_d$ be indeterminates. In this case we have 
equality in (\ref{eqn:diff-includ}), that is 
\[
I = \left\{g\in K[x_1,\ldots,x_d] :
g\left(\frac{\partial}{\partial t_1},
\cdots,\frac{\partial}{\partial t_d}\right)f(t_1\overline{x}_1 + \cdots + t_d\overline{x}_d) = 0\right\}.
\]
\end{proposition}

\begin{example}
Since $K$ has characteristic zero, the analytic complex functions $1/(1-z)$,
$e^z$ and $\sqrt{1+z}$ all have power series expansions with non-zero coefficients 
and would therefore work well as a candidate for our formal power series 
$f(z)$ in Proposition~\ref{prp:any-f}.
\end{example}

\section{Summary and further questions}
\label{sec:some-q}

We briefly discuss some of the main results in this article
and pose some relevant and motivating questions.

The novel contributions of this article are from 
Sections~\ref{sec:socle-IS}--\ref{sec:universal}. 
In Section~\ref{sec:socle-IS}, we first showed in an elementary combinatorial way 
 that every zero-dimensional monomial ideal $I$
can be given as the inverse 
ideal of the monomials contained in the docle $\doc(I)$ (Observation~\ref{obs:I=J}). This is a property 
unique to monomial ideals. We then derived in a combinatorial way the
double inverse system and double ideal in Corollary~\ref{cor:X-inv-inv}.
These double inverse operators are clearly closure operators. 

In Section~\ref{sec:closure-op}, we defined a new operator that is a closure 
operator with respect to a special partial order on the set of upsets of a poset, 
in Corollaries~\ref{cor:closure} and~\ref{cor:zero-dim=closed}.

In Section~\ref{sec:alg-observations}, we proved algebraically that 
zero-dimensional monomial ideals $I\subseteq J$ with $\doc(I)\subseteq\doc(J)$ 
must be equal (Proposition~\ref{prp:idealanddoc}). As a corollary, we then 
obtained an analogous statement for cofinite upsets in a poset in 
Corollary~\ref{cor:poset-idealanddoc}. The other main algebraic result of this section
is Theorem~\ref{thm:expandtomprimary} which allows us, together 
with~\cite[Theorem 1.1]{A-Epstein}, to present any upset $U$ with a nontrivial
$\doc(U)$ uniquely as an intersection $U = V\cap W$ of upsets where (a) 
$\doc(W) = \emptyset$ and (b) $V$ is cofinite. We then noted that this mentioned
uniqueness does not hold in the local case. 

The main result of Section~\ref{sec:Gor-inv} is Proposition~\ref{prp:P},
which was obtained by explicit computations with respect to the term order LEX.
This proposition justified us in defining the $d;k$-antipodal polynomial
$p_{d;k}^{{{\blacktriangle}}}$ of a homogeneous polynomial $p$, which is a monomial if and only if $p$ is a monomial. This made
it possible to present the if-and-only-if statement we sought for 
homogeneous zero-dimensional Gorenstein ideals in 
Corollary~\ref{cor:ann-iff-monomial}. 

Section~\ref{sec:universal}, the final contributing section of the paper, 
contains Proposition~\ref{prp:any-f}, which 
demonstrates that the polynomial $P$ from Corollary~\ref{cor:IS-CLS} can
be chosen rather freely. In other words, our homogeneous zero-dimensional Gorenstein ideal
has a more universal or general form in terms of analytic functions.

There are numerous questions left unanswered that would be good to resolve
at some point. As mentioned, Corollary~\ref{cor:ann-iff-monomial} only 
deals with the Gorenstein case. A natural question is therefore whether this
is a defining property of monomial ideals, namely the following question.
\begin{question}
\label{qst:non-Gorenstein}
Is it true that whenever we have a zero-dimensional non-Gorenstein ideal $I$ 
with $\soc(I)/I$ containing at least two elements, that $I$ is the inverse ideal 
of the elements of $\soc(I)/I$ if and only if $I$ is monomial?
\end{question}
The above question is also of interest if we restrict to homogeneous ideals.
\begin{question}
\label{qst:toric-ideal}
Is there an explicit and descriptive way to describe the socle of (homogeneous)
zero-dimensional toric ideals generated by binomials? Perhaps by the use of
a Gr\"{o}bner basis with respect to LEX, or some other well-chosen term order on 
$[x_1,\ldots,x_d]$?
\end{question}
As we have seen, we have been able to derive algebraic properties of monomial ideals
using purely combinatorial arguments on posets, their upsets and downsets and antichains.
Likewise we saw in Section~\ref{sec:alg-observations} that we were able to obtain
combinatorial results from algebraic ones using natural homomorphisms etc. 
\begin{question}
\label{qst:comb-proofs}
Is the class of monomial ideals the {\em only} large class of ideals of
$R = K[x_1,\ldots,x_d]$ where one can hope to be able to use combinatorial
arguments to obtain some interesting results about the socle and the docle $\doc(I)$
of an ideal? 
\end{question}
This last question might have a negative answer if Question~\ref{qst:toric-ideal} has
a positive answer. That is, toric ideals might be a class of ideals where one could
potentially derive many properties about the socle in purely combinatorial ways. 
Toric ideals have a rich connection to polytopes and possess a rich
combinatorial structure. Could this connection perhaps to polytopes be utilized somehow for
the socle of toric ideals?

\subsection*{Acknowledgments}  
Sincere thanks to the anonymous referees for their numerous excellent and pointed
suggestions to improve the paper. In particular the exposition of the paper has improved 
considerably by their help. Thanks to Justin Chen for pointing us to resources on 
primary decomposition over multigraded rings.

\bibliographystyle{amsalpha}
\bibliography{IS-Socles-arXiv-v2}

\end{document}